\documentclass[11pt]{article}
\usepackage[utf8]{inputenc}
\usepackage{geometry}
\usepackage{amssymb,amsthm,amsmath}
\usepackage[mathscr]{euscript}
\usepackage{graphicx}
\usepackage{algpseudocode}
\usepackage{algorithm}
\usepackage[english]{babel}
\usepackage{cleveref}
\usepackage{todonotes}
\usepackage{subfig}
\usepackage{pgfplots}
\usepackage{url}

\newtheorem{theorem}{Theorem}[section]
\newtheorem{lemma}[theorem]{Lemma}
\newtheorem{definition}[theorem]{Definition}

\newtheorem{prop}[theorem]{Proposition}

\newtheorem{remark}[theorem]{Remark}
\numberwithin{equation}{section}
\graphicspath{{./Groundtruthmaterials/}{./FBSmaterialrec/}{./ADMMmaterialrec/}{./FBSprimaldualmaterialrec/}}

\DeclareMathOperator*{\argmin}{argmin}
\DeclareMathOperator*{\argmax}{argmax}
\DeclareMathOperator*{\prox}{\mathbf{prox}}

\let\div\undefined
\DeclareMathOperator{\div}{div}

\newcommand{\vz}[1]{\ensuremath{\mathbb{#1}}}

\newcommand{\R}{{\vz R}}
\newcommand{\N}{{\vz N}}

\newcommand{\notinclude}[1]{}

\title{A non-convex variational model for joint polyenergetic CT reconstruction, sensor denoising and material decomposition\footnote{This work was supported by the European Union's Horizon 2020 research and innovation programme under grant agreement No.\ 777826, NoMADS, 
as well as the Deutsche Forschungsgemeinschaft (DFG, German Research Foundation) under Germany's Excellence Strategy -- EXC 2044 --, Mathematics M\"unster: Dynamics -- Geometry -- Structure, and under the Collaborative Research Centre 1450--431460824, InSight, University of M\"unster.}}
\author{Georgios Papanikos \\ Department of Mathematics, University of Manchester\\ \\ Benedikt Wirth\\ Department of Mathematics, University of M\"unster}

\date{}

\begin{document}
\maketitle

\begin{abstract}
Computed Tomography (CT) is widely used in engineering and medicine for imaging the interior of objects, patients, or animals.
If the employed X-ray source is monoenergetic, image reconstruction essentially means the inversion of a ray transform.
Typical X-ray sources are however polyenergetic (i.\,e.\ emit multiple wavelengths, each with different attenuation behaviour),
and ignoring this fact may lead to artefacts such as beam hardening.
An additional difficulty in some settings represents the occurrence of two different types of noise,
the photon counting effect on the detector and the electronic noise generated e.\,g.\ by CCD cameras.
We propose a novel variational image reconstruction model that takes both noise types and the polyenergetic source into account
and moreover decomposes the reconstruction into different materials based on their different attenuation behaviour.
In addition to a detailed mathematical analysis of the model we put forward a corresponding iterative algorithm including its convergence analysis.
Numerical reconstructions of phantom data illustrate the feasibility of the approach.
\\
\\
Key words: polyenergetic Computed Tomography, variational image reconstruction model, material decomposition, photon counting noise, electronic noise.  

\end{abstract}

\section{Introduction}\label{sec:intro}

Transmission X-ray Computed Tomography (CT) is a non-invasive imaging technique applied in many research areas such as engineering and medicine since it allows the visualization of internal structures of objects, humans, and animals \cite{markoe_2006,boerckel2014mct}. To this end X-rays are sent through the imaged object in different directions, and for each ray one records the remaining intensity behind the object. Typically one uses so-called polyenergetic or polychromatic X-ray sources that emit X-rays at different wave lengths. The corresponding image reconstruction can be formulated as a nonlinear inverse problem (see \cref{sec:ForPolyen} later). The advantage of using the nonlinear polyenergetic model instead of a simplified  monoenergetic (nonlinear or linear) model is that this prevents so-called beam hardening artefacts which otherwise will be present in the CT images as cupping and streaks.
Taking the polyenergetic nature of the X-ray source into account is particularly appropriate if one aims to recognize and reconstruct different materials within an object since each material absorbs or attenuates X-ray photons at different wavelengths in a characteristic way. Many works consider the material decomposition problem. Some of them perform the decomposition in two stages: For instance, they first reconstruct a scalar image from the measured projections (the so-called sinogram) and then decompose this image into different materials \cite{Gao_2011,Yao_2019}. Other approaches first decompose the sinogram into separate sinograms for the different materials and only afterwards reconstruct the materials \cite{HTom18}. The splitting between reconstruction and material decomposition is not optimal since it ignores the statistical errors (noise) in the measurements. 
Noise in transmission tomography is one of the main challenges for image reconstructiion, especially in low-dose tomography. The photon counting effect on the detector causes Poisson noise. In addition electronic noise or readout noise generated by other instruments like a CCD camera \cite{Modgil_2015,nuyts2013} can also occur in the final measurements. In the monoenergetic CT these two mixed noise contributions can be statistically modelled via a mixed Poisson and Gaussian distribution \cite{ding2018,nuyts2013}. In polyenergetic CT the statistical description of the noise is more challenging since the amount of Poisson noise depends on the arriving X-ray intensity at each wavelength. Rather than a Poisson distribution this results in a compound Poisson distribution (see \cite{nuyts2013} and the references therein).
A Bayesian ansatz can be used to develop so-called statistical iterative reconstruction techniques for modelling the noise. These techniques also have the advantage that they can incorporate prior
knowledge about the reconstructed object and can be easily adapted for material decomposition. Even though they can be very effective, most of the works assume a simplified noise model for the measurements such as pure Gaussian or pure Poisson noise \cite{Hu_2019,Landi_2017,JZang18, Liu_2016,schirra2014spectral,pereli2021}.

In this paper we propose a novel variational model for reconstructing and decomposing materials from X-ray CT measurements which accounts for a polyenergetic X-ray source and for photon counting Poisson noise at the different wavelenghts as well as additional electronic Gaussian noise. The proposed model is given by the minimization of the functional
\begin{align}\label{eq:mixedrecdenmodel}
F(y,\boldsymbol w)
=\frac1{2\sigma^2}\!\!\int_\Sigma\!\left(\!f\!-\!\int_{E_{\min}}^{E_{\max}}\!\!\!\!y\,\mathrm d\mathcal{L}^1\right)^{\!\!2}\!\mathrm d\mathcal H^n
\!+\!\int_\Sigma\int_{E_{\min}}^{E_{\max}}\!\!\!\!d_{\mathrm{KL}}(y,\mathscr I(\boldsymbol w))\,\mathrm d\mathcal{L}^1\,\mathrm d\mathcal H^n
+R^{\alpha,\beta}(\boldsymbol w)
\end{align}
over the material density function $\boldsymbol{w}$ and the wavelength-dependent sinograms $y$ (see \cref{sec:statmotBm}).
In more detail, $\boldsymbol w(\boldsymbol x)\in\R^{N_{\text{mat}}}$ indicates the fractions of the $N_{\text{mat}}$ different materials at position $\boldsymbol x$,
while $y(E,\boldsymbol x,\boldsymbol\theta)$ describes the number of photons of wavelength (or photon energy) $E$ arriving at a position $(\boldsymbol x,\boldsymbol\theta)$ in the sinogram domain.
The measurement $f(\boldsymbol x,\boldsymbol\theta)$, in contrast, does not distinguish different wavelengths.
Furthermore, $\sigma\in\R$ denotes the standard deviation of the electronic Gaussian noise, $d_{\mathrm{KL}}$ denotes the so-called Kullback--Leibler divergence,
$\mathscr I(\boldsymbol w)$ denotes the forward operator (turning the material distribution $\boldsymbol w$ into the expected photon intensity at every wavelength and sinogram position).
The first two integrals act as data discrepancy terms, while $R^{\alpha,\beta}(\boldsymbol{w})$ represents a regularization functional with regularization weights $\alpha,\beta>0$
(we will use total variation and a multiwell term promoting pure materials). Our data fidelity can be viewed as an extension of the data fidelity terms proposed initially in  \cite{calatroni2017} and then adapted in \cite{Gpap2020} for monoenergetic CT reconstruction.

Our first main contribution is to provide a theoretical analysis of \eqref{eq:mixedrecdenmodel} proving the existence of minimisers using the direct methods of the calculus of variation. The second contribution is the design of a corresponding iterative scheme based on the works of \cite{brune2010,sawatzky2011} and \cite{Gpap2020} which concern image reconstruction from positron emission tomography (PET) measurements and monoenergetic transmission tomography measurements with Poisson and mixed Gaussian-Poisson noise, respectively. We then rigorously analyse our scheme and prove its convergence to a critical point of \eqref{eq:mixedrecdenmodel}. The proposed algorithm is efficient in the regime of medium to strong regularization, while for low regularization (in case of very low noise) it may require too many iterations, which is why we also provide two alternative primal-dual-based algorithms in the appendix. 

This paper is structured as follows. In \cref{sec:ForPolyen} we first present the forward operator and its mathematical properties, after which we derive our variational model via a Bayesian ansatz in \cref{sec:statmotBm} and prove its well-posedness in \cref{sec:Analysis}. In \cref{sec:FBSalg} we propose an iterative scheme and detail all its algorithmic ingredients, closing the section with a stability and convergence analysis. Finally, \cref{sec:Numresults} presents numerical results on phantom data to illustrate the efficacy of our model and algorithm for polyenergetic CT measurements.
 
We will denote $\R^n$-valued quantities or functions by boldfont letters.
The $n$-dimensional Lebesgue measure is denoted $\mathcal L^n$, the $n$-dimensional Hausdorff measure $\mathcal H^n$.
For $p\in[1,\infty]$ the corresponding Lebesgue space on an open bounded domain $\Omega$ is denoted $L^p(\Omega)$,
where sometimes we indicate a range $R$ by $L^p(\Omega;R)$.
The space of functions of bounded variation is denoted $\mathrm{BV}(\Omega)$ with corresponding total variation seminorm $\mathrm{TV}$.
Finally, the proximal operator of some function $f$ on a Hilbert space $X$ will be denoted by $\prox_{f}(x)=\argmin_{y\in X}\frac12\|x-y\|^2_X+f(y)$,
and its Legendre--Fenchel conjugate by $f^*(y)=\sup_{x\in X}(x,y)_X-f(y)$,
where $(\cdot,\cdot)_X$ and $\|\cdot\|_X$ denote the inner product and norm on $X$.


\section{A variational model for multimaterial polyenergetic CT reconstruction}\label{sec:statmotBm}
This section presents, motivates and briefly analyses the variational model we propose.

\subsection{The polyenergetic CT forward operator}\label{sec:ForPolyen}
We briefly recapitulate the polyenergetic CT forward operator.
The imaged sample is supposed to lie in the bounded open domain $\Omega\subset\R^n$ ($n=2,3$).
In practical X-ray CT the pointlike X-ray source is moved around the sample along a smooth curve $\boldsymbol{\xi}\subset\R^n\setminus\overline\Omega$ of finite length and without self-intersections.
When the source is at position $\boldsymbol x\in\boldsymbol\xi$, the X-rays emanating from it in direction $\boldsymbol\theta\subset\mathbb{S}^{n-1}$ correspond to the half line
\begin{equation*}
L^{+}_{\boldsymbol{x},\boldsymbol{\theta}}:=\{\boldsymbol{x}+s\boldsymbol\theta: \ s\in [0,\infty)\}.
\end{equation*}
The X-ray intensity transmitted through $\Omega$ along each half line is recorded, thus we have a measurement for every point in $\boldsymbol\xi\times\mathbb S^{n-1}$,
where without loss of generality we may also reduce $\boldsymbol\xi\times\mathbb S^{n-1}$ to
\begin{equation*}
\Sigma=\{ (\boldsymbol{x},\boldsymbol{\theta})\in \boldsymbol\xi\times \mathbb{S}^{n-1} :  \ \mathcal{H}^1(L^{+}_{\boldsymbol{x},\boldsymbol{\theta}}\cap \Omega)>0\},
\end{equation*}
equipped with the $n$-dimensional Hausdorff measure $\mathcal H^n$.
The associated linear operator is the so-called divergent X-ray transform \cite{markoe_2006}, also known as fan- or cone-beam transform in two and three space dimensions, respectively.
\begin{definition}[Divergent X-ray transform]\label{def:divrgXray}
	The \emph{divergent X-ray transform} is the bounded linear operator $\mathscr{D}:L^q(\Omega)\to L^q(\Sigma)$ for $q\in[1,\infty]$ defined by
	\begin{equation}
	\mathscr{D}u(\boldsymbol{x},\boldsymbol{\theta}) = 
	\int_{L^{+}_{\boldsymbol{x},\boldsymbol{\theta}}} u(\boldsymbol{z})\, \mathrm{d}\mathcal{H}^{1}(\boldsymbol{z}).
	\end{equation}
\end{definition}
The following properties are straightforward to check.


\begin{lemma}[Properties of $\mathscr{D}$]\label{assm: PropofDD*}
	\mbox{}
	\begin{enumerate}
	\item The dual operator $\mathscr{D}^*:L^{q'}(\Sigma)\rightarrow L^{q'}(\Omega)$ to $\mathscr{D}$ with $q'=\frac q{q-1}$ is given by
	\begin{equation*}\label{eq:duadivXray}
	\mathscr{D}^*w(\boldsymbol{z})= \int_{\boldsymbol\xi}|\boldsymbol{z}-\boldsymbol{x}|^{1-n} w\left(\boldsymbol{x},\frac{\boldsymbol{z}-\boldsymbol{x}}{|\boldsymbol{z}-\boldsymbol{x}|}\right) \,\mathrm d\mathcal{H}^1(\boldsymbol{x}).
	\end{equation*} 
		\item \label{item:A2 positivityofD} The operator $\mathscr{D}$ preserves positivity, i.e., $\mathscr{D}u\geq 0$ for any $u\geq 0$.
		\item The divergent beam transform $\mathscr D$ is also bounded as an operator from $L^\infty(\Omega)$ to $L^\infty(\Sigma)$ with norm no larger than the diameter of $\Omega$.
		\item The operator $\mathscr{D}$ does not annihilate constant functions, i.e., $\mathscr{D}\boldsymbol{1}_{\Omega}\neq 0$.
	\end{enumerate}
\end{lemma}   

When passing through a material, an X-ray is attenuated.
The amount of attenuation at each point is quantified by the (spatially varying, material-dependent) attenuation coefficient $u:\Omega\to[0,\infty)$.
By Beer's law the energy fraction passing through the material without being absorbed (the so-called transmittance) equals $\exp(-U)$
for $U=\int_{L^+_{\boldsymbol{x},\boldsymbol{\theta}}}u(x)\,\mathrm d\mathcal H^1$ the accumulated attenuation coefficient along the ray.
However, this formula only holds for monochromatic or monoenergetic X-rays, since the attenuation coefficient of a material also depends on the X-ray wavelength.
Practical X-ray sources are polyenergetic, i.e., they emit a continuous spectrum of wave lengths, described by an integrable function
\begin{equation*}
I_0:[E_{\min},E_{\max}]\to[0,\infty)
\end{equation*}
that indicates the intensity or amount of X-rays emitted at each wavelength
(rather than the wavelength, the argument of $I_0$ is the energy of a photon, which is in one-to-one correspondence with the photon's wavelength
and lies between a minimum and a maximum possible energy $E_{\min}$ and $E_{\max}$).
Thus, if $u:\Omega\times[E_{\min},E_{\max}]\to[0,\infty)$ describes the attenuation coefficient as a function of spatial position and photon energy,
the expected number of photons or X-ray intensity passing through the material is expressed by the following polyenergetic CT forward operator.

\begin{definition}[Polyenergetic CT forward operator]\label{def: polyen Intensity operator}
	The \emph{expected X-ray intensity} resulting from a Lebesgue measurable $u:\Omega\times[E_{\min},E_{\max}]\to[0,\infty)$ is defined for $\mathcal L^1\otimes\mathcal H^n$-almost all $(E,\boldsymbol x,\boldsymbol\theta)\in[E_{\min},E_{\max}]\times\Sigma$ via
	\begin{equation*}
	\widetilde{\mathscr{I}}(u)(E,\boldsymbol x,\boldsymbol\theta) =I_0(E) \exp\left(-\mathscr{D}u(\boldsymbol x,\boldsymbol\theta,E) \right),
	\end{equation*}
	where for each $E>0$ we write $\mathscr{D}u(\boldsymbol x,\boldsymbol\theta,E)$ for $\mathscr D(u(\cdot,E))$ evaluated at $(\boldsymbol x,\boldsymbol\theta)$.
	The \emph{polyenergetic CT forward operator} then is
	\begin{equation*}\label{eq:Cpolyen}
	\widetilde{\mathscr{F}}(u)(\boldsymbol x,\boldsymbol\theta)
	=\int_{E_{\min}}^{E_{\max}}\widetilde{\mathscr{I}}(u)(E,\boldsymbol x,\boldsymbol\theta)\,\mathrm{d}\mathcal{L}^1(E)
	=\int_{E_{\min}}^{E_{\max}} I_0(E) \exp\left(-\mathscr{D}u(\boldsymbol x,\boldsymbol\theta,E) \right)\,\mathrm{d}\mathcal{L}^1(E).
	\end{equation*}
\end{definition}

Since in the above forward operator the transmittance is accumulated over all photon energies
(unlike in spectral CT, where photons at different wavelengths are distinguished in the measurement),
it is impossible to reconstruct the energy dependence of the attenuation coefficient $u$.
Therefore, the use of the polyenergetic CT forward operator only makes sense in combination with further prior knowledge.
We will assume that the scanned object consists of $N_{\text{mat}}$ different materials with known, bounded photon energy dependent attenuation coefficients
\begin{equation*}
g_i:[E_{\min},E_{\max}]\to[0,\infty),\quad i=1,\ldots,N_{\text{mat}}
\end{equation*}
as they can for instance be found in \url{www.nist.gov/pml/
	data/xraycoef/}.
We will write $\boldsymbol g=(g_1,\ldots,g_{N_{\text{mat}}})$.
The attenuation coefficient of the sample can then be expressed as
\begin{equation}\label{eq:matdec}
u(\boldsymbol{z},E)=\sum_{i=1}^{N_{\text{mat}}}g_i(E)w_i({\boldsymbol{z}})
\end{equation}
with $w_i:\Omega\to[0,1]$ the spatial distribution of the $i$th material, indicating at each point its volume proportion.
Including void as one of the materials, all volume proportions have to add up to one,
\begin{equation*}
\sum_{i=1}^{N_{\text{mat}}}w_i=1.
\end{equation*}
Summarizing, the expected intensity of X-rays passing through the sample can be expressed in terms of the unknown material distributions $\boldsymbol w=(w_1,\ldots,w_{N_{\text{mat}}})$ as
\begin{equation*}
\mathscr{I}(\boldsymbol w)(E,\boldsymbol x,\boldsymbol\theta)
=\widetilde{\mathscr{I}}(u)(E,\boldsymbol x,\boldsymbol\theta)
=I_0(E) \exp\left(-\sum_{i=1}^{N_{\text{mat}}}g_i(E)\mathscr{D}w_i(\boldsymbol x,\boldsymbol\theta)\right),
\end{equation*}
and we aim to reconstruct $\boldsymbol w$ from a measurement of
\begin{equation}\label{eqn:polyenForwardOp}
\mathscr{F}(\boldsymbol w)
=\widetilde{\mathscr{F}}(u)
=\int_{E_{\min}}^{E_{\max}} I_0(E) \exp\left(-\sum_{i=1}^{N_{\text{mat}}}g_i(E)\mathscr{D}w_i\right)\,\mathrm{d}\mathcal{L}^1(E).
\end{equation}

%
%

\begin{lemma}[$L^1$-continuity of $\mathscr I$ and $\mathscr F$]\label{lem:L1Continuity of Cpoly}
The operator $\mathscr I$ is continuous from $L^1(\Omega;[0,1]^{N_{\text{mat}}})$ to $L^1((E_{\min},E_{\max})\times\Sigma)$,
the operator $\mathscr F$ is continuous from $L^1(\Omega;[0,1]^{N_{\text{mat}}})$ to $L^1(\Sigma)$.
\end{lemma}
\begin{proof}
Since $w_i,g_i\geq0$ for all $i$ we have $\exp\left(-\sum_{i=1}^{N_{\text{mat}}}g_i(E)\mathscr{D}w_i\right)\in[0,1]$ and thus
\begin{equation*}
|\mathscr I(\boldsymbol w)(E,\boldsymbol x,\boldsymbol\theta)|\leq I_0(E)
\qquad\text{and}\qquad
|\mathscr F(\boldsymbol w)|
\leq\int_{E_{\min}}^{E_{\max}} I_0(E)\,\mathrm{d}\mathcal{L}^1(E)
\end{equation*}
so that $\mathscr I$ indeed maps into $L^1((E_{\min},E_{\max})\times\Sigma)$ and $\mathscr F$ even maps into $L^\infty(\Sigma)$.
Furthermore, due to the Lipschitz continuity of $\exp$ on $(-\infty,0]$ with Lipschitz constant $1$ we have
\begin{multline*}
|\mathscr I(\boldsymbol w)-\mathscr I(\tilde{\boldsymbol w})|
\leq I_0(E) \left|\exp\left(-\sum_{i=1}^{N_{\text{mat}}}g_i(E)\mathscr{D}w_i\right)-\exp\left(-\sum_{i=1}^{N_{\text{mat}}}g_i(E)\mathscr{D}\tilde w_i\right)\right|\\
\leq I_0(E) \sum_{i=1}^{N_{\text{mat}}}g_i(E)|\mathscr{D}w_i-\mathscr{D}\tilde w_i|
\leq\max_i\left(I_0(E)g_i(E)\right)\sum_{i=1}^{N_{\text{mat}}}|\mathscr{D}w_i-\mathscr{D}\tilde w_i|.
\end{multline*}
Integrating over $(E_{\min},E_{\max})\times\Sigma$, the continuity statement for $\mathscr I$ then follows from the $L^1$-continuity of $\mathscr D$.
The $L^1$-continuity of $\mathscr F$ now is a direct consequence due to
$\|\mathscr F(\boldsymbol w)-\mathscr F(\tilde{\boldsymbol w})\|_{L^1}
=\big\|\int_{E_{\min}}^{E_{\max}}\mathscr I(\boldsymbol w)-\mathscr I(\tilde{\boldsymbol w})\,\mathrm d\mathcal L^1\big\|_{L^1}
\leq\|\mathscr I(\boldsymbol w)-\mathscr I(\tilde{\boldsymbol w})\|_{L^1}$.
\end{proof}

\subsection{Bayesian motivation of the variational reconstruction}\label{sec:statmotBm}
We now briefly motivate our variational approach for the CT reconstruction from a Bayesian perspective, following \cite{calatroni2017,Gpap2020,lanza2014}.
The quantity $\mathscr{F}(\boldsymbol w)(\boldsymbol x,\boldsymbol\theta)$ from \eqref{eqn:polyenForwardOp}
in fact only represents the expected intensity of X-rays passing through the sample along the ray defined by $(\boldsymbol x,\boldsymbol\theta)$.
The actual number of arriving photons, though, which is the measured quantity, is governed by a Poisson process with that intensity.
In addition, the photons are typically counted with charged coupled devices (CCD) that exhibit Gaussian background noise\cite{Modgil_2015,markoe_2006}.
We aim to account for both types of noise.

Let the domain $\Sigma$ of the measurement be partitioned into detectors $D_i$
(in fact, in a real CT system the physical detectors are moved along with the X-ray source,
so each $D_i$ corresponds to a true physical detector, representing a subset of $\mathbb S^{n-1}$,
times a segment of the curve $\boldsymbol\xi$ corresponding to a measurement time interval during the motion).
Similarly, let the energy interval $[E_{\min},E_{\max}]$ be partitioned into subintervals $E_i$.
The expected number of photons from energy range $E_i$ arriving at detector $D_j$ is then given by
\begin{equation*}
\bar y_{ij}(\boldsymbol w)
=\int_{E_i}\int_{D_j}\mathscr{I}(\boldsymbol w)(E,\boldsymbol x,\boldsymbol\theta)\,\mathrm{d}\mathcal{H}^n(\boldsymbol x,\boldsymbol\theta)\,\mathrm{d}\mathcal{L}^1(E).
\end{equation*}
The actual number of photons $y_{ij}$ then is a Poisson-distributed random variable with mean $\bar y_{ij}(\boldsymbol w)$,
\begin{equation*}
y_{ij}\sim\mathcal P(\bar y_{ij}(\boldsymbol w)),
\end{equation*}
and the actual measurement, the readout $f_j$ of detector $D_j$, will be the total number of photons, summed over all energy ranges and perturbed by normally distributed noise, thus
\begin{equation*}\textstyle
f_{j}\sim\mathcal N\left(\sum_i y_{ij},\mathcal H^n(D_j)\sigma^2\right),
\end{equation*}
where $\sigma^2$ denotes the (usually known) noise variance per detector area.
Note that since the noise of all CCDs within a detector $D_j$ is independent, the noise variance of the detector is the sum of the noise variances of each CCD by the Bienaym\'e formula,
and since the number of CCDs is proportional to the detector area, so must be the noise variance.

The quantities $f=(f_j)_j$ and $y=(y_{ij})_{ij}$ are stochastic variables,
whose distribution depends on $\bar y(\boldsymbol w)=(\bar y_{ij}(\boldsymbol w))_{ij}$ and thus on $\boldsymbol w$.
Denoting by $\pi(y,\bar y(\boldsymbol w)|f)$ the probability density of true photon numbers $y$ and expected intensity $\bar y(\boldsymbol w)$ given the measurement $f$,
the maximum a posteriori (MAP) estimate tries to identify $y$ and $\bar y(\boldsymbol w)$ or rather $\boldsymbol w$ as the configuration with highest probability,
\begin{equation*}
(y,\boldsymbol w)=\argmax\pi(y,\bar y(\boldsymbol w)|f)=\argmin-\log\pi(y,\bar y(\boldsymbol w)|f).
\end{equation*}
By Bayes' rule,
\begin{equation*}
\pi(y,\bar y(\boldsymbol w)|f)
=\frac{\pi(y,\bar y(\boldsymbol w),f)}{\pi(f)}
=\frac{\pi(y,\bar y(\boldsymbol w),f)}{\pi(y,\bar y(\boldsymbol w))}\frac{\pi(y,\bar y(\boldsymbol w))}{\pi(\bar y(\boldsymbol w))}\frac{\pi(\bar y(\boldsymbol w))}{\pi(f)}
=\frac{\pi(f|y,\bar y(\boldsymbol w))\pi(y|\bar y(\boldsymbol w))\pi(\bar y(\boldsymbol w))}{\pi(f)}.
\end{equation*}
Due to the above assumptions on the distributions we can explicitly express two of the probability densities,
\begin{align*}
\pi(f|y,\bar y(\boldsymbol w))
&=\prod_j\pi(f_j|(y_{ij})_i)
=\prod_j\frac{1}{\sqrt{2\pi\mathcal H^n(D_j)}\sigma}\exp\left(-\frac{1}{2\mathcal H^n(D_j)\sigma^2}\left(f_j-\sum_i y_{ij}\right)^2\right),\\
\pi(y|\bar y(\boldsymbol w))
&=\prod_{i,j}\pi(y_{ij}|\bar y_{ij}(\boldsymbol w))
=\prod_{i,j}\frac{\bar y_{ij}(\boldsymbol w)^{y_{ij}}e^{-\bar y_{ij}(\boldsymbol w)}}{y_{ij}!}.
\end{align*}
Furthermore, as is frequently done, for the prior density $\pi(\bar y(\boldsymbol{w}))$ we will assume a Gibbs distribution of the form
\begin{equation*}
\pi(\bar y(\boldsymbol{w}))
=\pi(\boldsymbol{w})
=\frac1Z\exp(-R^{\alpha,\beta}(\boldsymbol{w}))
\end{equation*}
with $Z$ a normalization constant and $R^{\alpha,\beta}$ a regularization functional depending on nonnegative tuning weights $\alpha,\beta\geq0$.
Material distributions with lower values of $R^{\alpha,\beta}$ are thus assigned a higher probability.
Abbreviating the probability simplex as
\begin{equation*}
\Delta=\left\{\boldsymbol w\in\R^{N_{\text{mat}}}\,\middle|\,w_1,\ldots,w_{N_{\text{mat}}}\geq0,\,\sum_{i=1}^{N_{\text{mat}}}w_i=1\right\},
\end{equation*}
we choose
\begin{equation*}
R^{\alpha,\beta}(\boldsymbol w)
=\alpha R_1(\boldsymbol w)+\beta R_2(\boldsymbol w)
\quad\text{with }
\begin{cases}
R_1(\boldsymbol w)=\sum_{i=1}^{N_{\text{mat}}}\mathrm{TV}(w_i)+\int_\Omega\iota_\Delta(\boldsymbol w)\,\mathrm{d}\mathcal{L}^n
\text{ and }\\
R_2(\boldsymbol w)=-\frac12\int_\Omega|\boldsymbol w-\mathbf1/N_{\text{mat}}|^2\,\mathrm{d}\mathcal{L}^n.
\end{cases}
\end{equation*}
Above, $\mathbf1$ represents the vector of ones, $\iota_\Delta$ is the convex indicator function of the probability simplex, and $\mathrm{TV}$ is the total variation seminorm.
The total variation term in $R_1$ is known to prefer spatially piecewise constant material densities, while the quadratic term $R_2$ is minimal at pure materials and thus prefers these.

Summarizing, if in $-\log\pi(y,\bar y(\boldsymbol w)|f)$ we ignore terms independent of $(y,\boldsymbol w)$
(which we will generally do in the following as they have no influence on the minimizer),
the MAP estimate is obtained as minimizer $(y,\boldsymbol w)$ of
\begin{equation*}
\sum_j\frac{1}{2\mathcal H^n(D_j)\sigma^2}\left(f_j-\sum_i y_{ij}\right)^2
+\sum_{i,j}\left(\log y_{ij}!+\bar y_{ij}(\boldsymbol w)-y_{ij}\log\bar y_{ij}(\boldsymbol w)\right)
+R^{\alpha,\beta}(\boldsymbol w).
\end{equation*}
With the Stirling approximation $\log(n!)= n\log(n)-n+O(\log(n))$ as $n\to \infty$ and the Kullback--Leibler divergence
\begin{equation*}
d_{\mathrm{KL}}(z,\bar z)
=\begin{cases}
z\log(\tfrac z{\bar z})-z+\bar z&\text{if }z,\bar z\geq0\\
\infty&\text{else}
\end{cases}
\end{equation*}
(using the formal convention $\frac10=\infty$, $\frac00=1$, and $0\log0=0$) this becomes
\begin{equation*}
\sum_j\frac{1}{2\mathcal H^n(D_j)\sigma^2}\left(f_j-\sum_i y_{ij}\right)^2
+\sum_{i,j}d_{\mathrm{KL}}(y_{ij},\bar y_{ij}(\boldsymbol w))
+R^{\alpha,\beta}(\boldsymbol w).
\end{equation*}
Now identify the vectors $f$ and $y$ with piecewise constant functions $f:\Sigma\to\R$ and $y:[E_{\min},E_{\max}]\times\Sigma\to\R$
(with a slight misuse of notation we use the same symbol) such that
\begin{equation*}
f=\frac{f_j}{\mathcal{H}^n(D_j)}\text{ on detector $D_j$, }\qquad
y=\frac{y_{ij}}{\mathcal{L}^1(E_i)\mathcal{H}^n(D_j)}\text{ on $E_i\times D_j$}
\end{equation*}
(thus $f_j$ and $y_{ij}$ are the accumulated values of $f$ on $D_j$ and $y$ on $E_i\times D_j$).
Furthermore note that under sufficient regularity of $\boldsymbol w$ (for instance $\boldsymbol w$ with strictly convex level lines)
the intensity $\mathscr{I}(\boldsymbol w)$ is continuous so that
$\mathscr{I}(\boldsymbol w)\approx\frac{\bar y_{ij}(\boldsymbol w)}{\mathcal{L}^1(E_i)\mathcal{H}^n(D_j)}$ on $E_i\times D_j$.
With this approximation the above energy now reads
\begin{equation*}
F(y,\boldsymbol w)
=\frac1{2\sigma^2}\int_\Sigma\left(f-\int_{E_{\min}}^{E_{\max}}y\,\mathrm d\mathcal{L}^1\right)^2\,\mathrm d\mathcal H^n
+\int_\Sigma\int_{E_{\min}}^{E_{\max}}d_{\mathrm{KL}}(y,\mathscr I(\boldsymbol w))\,\mathrm d\mathcal{L}^1\,\mathrm d\mathcal H^n
+R^{\alpha,\beta}(\boldsymbol w).
\end{equation*}
Therefore, we will define our CT-reconstruction from the measurement $f$ as a minimizer of this energy $F$
(where we no longer require $y$ to be piecewise constant).
We remark that the last transition from discrete detectors to the continuum setting has been merely formal,
but could be made rigorous in form of a $\Gamma$-convergence result,
showing that minimizers of the functional with discrete detectors converge to minimizers of $F$ as the size of the detectors and distinguished energy ranges uniformly tends to zero.

\subsection{Well-posedness analysis}\label{sec:Analysis}
We briefly prove existence of minimizers for our reconstruction functional $F$.
Note that $F$ is in general not convex so that uniqueness of minimizers cannot be expected:
Indeed, the $\beta$-weighted regularization term in $R^{\alpha,\beta}$ is purposely designed to have different local minima exactly at the pure material distributions.
Furthermore, even though the Kullback--Leibler divergence is convex, its composition with the nonlinear operator $\mathscr I$ no longer is,
as can easily be seen from $(y,s)\mapsto d_{\mathrm{KL}}(y,I_0(E)\exp(-s))$ having Hessian
\begin{equation*}
\begin{pmatrix} 1/y & 1 \\ 1 & I_0(E)\exp(-s)\end{pmatrix},
\end{equation*}
which may become indefinite for large $y$ or $s$.
Another consequence of the nonconvexity is that for existence of minimizers strong compactness in $L^1$ will be required of material distributions with finite cost $F$.
This strong compactness is provided by the total variation seminorm within the regularizer $R^{\alpha,\beta}$, without which the reconstruction problem would not be well-posed.
We will exploit that for given $\boldsymbol w$ the optimal $y$ can be explicitly computed.

\begin{lemma}[Minimizing photon count]\label{thm:optimalY}
For any $f\in L^2(\Sigma)$ and $\boldsymbol w\in (L^\infty(\Omega))^{N_{\text{mat}}}$ the minimizer $y_{\boldsymbol w}$ of $F(\cdot,\boldsymbol w)$ can be explicitly computed pointwise.
Indeed, abbreviating $Y_{\boldsymbol w}=\int_{E_{\max}}^{E_{\min}}y_{\boldsymbol w}\,\mathrm d\mathcal L^1$, the functions $y_{\boldsymbol w}$ and $Y_{\boldsymbol w}$ satisfy
\begin{align*}
y_{\boldsymbol w}
&\textstyle=\mathscr I(\boldsymbol w)\exp\left(\frac{f-Y_{\boldsymbol w}}{\sigma^2}\right),\\
Y_{\boldsymbol w}
&\textstyle=\mathscr F(\boldsymbol w)\exp\left(\frac{f-Y_{\boldsymbol w}}{\sigma^2}\right)\\
&\textstyle=\sigma^2W_0\left(\frac1{\sigma^2}\mathscr F(\boldsymbol w)\exp\left(\frac{f}{\sigma^2}\right)\right)
\end{align*}
for $W_0$ the Lambert-W function which is the inverse of $z\mapsto z\exp(z)$ (cf.\ \cite{borwein2016}).
\end{lemma}
\begin{proof}
The functional $F$ is strictly convex in $y$.
Thus it suffices to solve the first order optimality condition
\begin{equation*}
0=\partial_yF(y_{\boldsymbol w},\boldsymbol w)
=\frac1{\sigma^2}\left(\int_{E_{\min}}^{E_{\max}}y_{\boldsymbol w}\,\mathrm d\mathcal L^1-f\right)+\log\frac{y_{\boldsymbol w}}{\mathscr I(\boldsymbol w)}
=\frac1{\sigma^2}\left(Y_{\boldsymbol w}-f\right)+\log\frac{y_{\boldsymbol w}}{\mathscr I(\boldsymbol w)},
\end{equation*}
which yields the first equation.
Integrating both sides of that equation over $[E_{\min},E_{\max}]$
and using $\mathscr F(\boldsymbol w)=\int_{E_{\min}}^{E_{\max}}\mathscr I(\boldsymbol w)\,\mathrm d\mathcal L^1$ yields the second equation.
Multiplying that second equation by $\frac1{\sigma^2}\exp\frac{Y_{\boldsymbol w}}{\sigma^2}$ yields
\begin{equation*}
\frac{Y_{\boldsymbol w}}{\sigma^2}\exp\frac{Y_{\boldsymbol w}}{\sigma^2}
=\frac1{\sigma^2}\mathscr F(\boldsymbol w)\exp\left(\frac{f}{\sigma^2}\right)
\end{equation*}
so that indeed $Y_{\boldsymbol w}=\sigma^2W_0\big(\frac1{\sigma^2}\mathscr F(\boldsymbol w)\exp\big(\frac{f}{\sigma^2}\big)\big)$.
\end{proof}

Now we can show existence of minimizers for $F$.

\begin{theorem}[Existence of minimizers]
For any $f\in L^2(\Sigma)$ there exists a minimizer $(y,\boldsymbol w)\in L^\infty((E_{\min},E_{\max})\times\Sigma)\times(\mathrm{BV}(\Omega))^{N_{\text{mat}}}$ of $F$ with finite $F(y,\boldsymbol w)$.
\end{theorem}
\begin{proof}
We follow the standard direct method of the calculus of variations.
The choice $y=0$, $w_1=1$, $w_2=\ldots=w_{N_{\text{mat}}}=0$ satisfies
$F(y,\boldsymbol w)=\frac1{2\sigma^2}\|f\|_{L^2}^2+\int_\Sigma\int_{E_{\max}}^{E_{\min}}\mathscr I(\boldsymbol w)\,\mathrm d\mathcal L^1\,\mathrm d\mathcal H^n<\infty$
so that $F$ is proper.
Furthermore, since $d_{\mathrm{KL}}\geq0$ we have $F(y,\boldsymbol w)\geq-\beta\mathcal L^n(\Omega)N_{\text{mat}}(N_{\text{mat}}-1)$ so that $F$ is bounded from below.

Now consider a minimizing sequence $(y^n,\boldsymbol w^n)$, $n=1,2,\ldots$, such that $F(y^n,\boldsymbol w^n)$ converges monotonically to $\inf F$.
Without loss of generality we may assume $F(y^1,\boldsymbol w^1)<\infty$
so that for all $n$ we have $w_1^n,\ldots,w_{N_{\text{mat}}}^n\geq0$ with $w_1^n+\ldots+w_{N_{\text{mat}}}^n=1$ almost everywhere (otherwise $R^{\alpha,\beta}$ would be infinite).
Furthermore we may assume $y^n=y_{\boldsymbol w^n}$ to be the optimal photon count from \cref{thm:optimalY}.

Since $\alpha\mathrm{TV}(w_i^n)$ is part of $F$, all $w_1^n,\ldots,w_{N_{\text{mat}}}^n$ are uniformly bounded in $\mathrm{BV}(\Omega)$
so that there exists a weakly-* converging subsequence $\boldsymbol w^n\stackrel*\rightharpoonup\overline{\boldsymbol w}$ in $(\mathrm{BV}(\Omega))^{N_{\text{mat}}}$,
which for simplicity we again index by $n$.
By the compact embedding of $BV$ into $L^1$ we also have strong convergence in $L^1(\Omega)$
and thus by \cref{lem:L1Continuity of Cpoly} that
$\mathscr I(\boldsymbol w^n)\to\mathscr I(\overline{\boldsymbol w})$ in $L^1((E_{\min},E_{\max})\times\Sigma)$
and  $\mathscr F(\boldsymbol w^n)\to\mathscr F(\overline{\boldsymbol w})$ in $L^1(\Sigma)$.
After extracting a further subsequence we may in both cases even assume pointwise convergence almost everywhere.
\Cref{thm:optimalY} then implies pointwise convergence almost everywhere
of $Y_{\boldsymbol w^n}$ and $y^n=y_{\boldsymbol w^n}$ to $Y_{\overline{\boldsymbol w}}$ and $\bar y=y_{\overline{\boldsymbol w}}$.
We thus obtain
\begin{multline*}
\lim_{n\to\infty}F(y^n,\boldsymbol w^n)
=\lim_{n\to\infty}\frac1{2\sigma^2}\int_\Sigma\left(f-Y_{\boldsymbol w^n}\right)^2\,\mathrm d\mathcal H^n
+\int_\Sigma\int_{E_{\min}}^{E_{\max}}d_{\mathrm{KL}}(y_{\boldsymbol w^n},\mathscr I(\boldsymbol w^n))\,\mathrm d\mathcal{L}^1\,\mathrm d\mathcal H^n\\
+\alpha\sum_{i=1}^{N_{\text{mat}}}\mathrm{TV}(w_i^n)
+\frac\beta2\int_\Omega-|\boldsymbol w^n-\mathbf1/N_{\text{mat}}|^2\,\mathrm{d}\mathcal{L}^n
\geq F(y_{\overline{\boldsymbol w}},\overline{\boldsymbol w}),
\end{multline*}
where the last inequality follows from the sequential lower semi-continuity of $\mathrm{TV}$ with respect to weak-* convergence in $BV$
and from Fatou's lemma since all integrands are lower semi-continuous and uniformly bounded from below.
Thus a minimizer is given by $(y_{\overline{\boldsymbol w}},\overline{\boldsymbol w})$.
\end{proof}

\section{An alternating minimization with EM type update}\label{sec:FBSalg}
\notinclude{\todo[inline]{need to change introduction, probably along the following lines: ADMM is possible, too, and method only applicable in case of high noise and large regularization; advantage: method provably convergent; we will propose method and prove convergence}
The energy $F(y,\boldsymbol{w})$ to be minimized consists of separate summands,
each of which turns out to have a rather simple proximal operator in the variables $y$ and $\boldsymbol w$ (note that for $\beta=0$ the cost $F$ is even convex separately in $y$ and $\boldsymbol w$).
Therefore, the most obvious minimization algorithm would be a standard splitting approach
in which each summand is treated independently
by introducing multiple copies of $y$ and $\boldsymbol w$ (one for each summand)
as well as the auxiliary constraint that these copies coincide.
The minimization in such a splitting approach is typically tackled by the alternating direction method of multipliers (ADMM).
However, in our numerical experiments the ADMM turned out to be instable and did not converge, which we associated with the nonconvexity of $F$.
Resorting to a full augmented Lagrangian method instead (of which the ADMM is a simplification) led to an infeasible computation time for each iteration.
Consequently, we here present an alternating minimization alorithm for $y$ and $\boldsymbol w$
in which the minimization of $\boldsymbol w$ is of forward backward splitting type and does not introduce auxiliary copies of $\boldsymbol w$.
}
We here present an alternating minimization alorithm for $y$ and $\boldsymbol w$
in which the minimization of $\boldsymbol w$ is of forward backward splitting type.
The scheme is provably convergent and, through an update similar to expectation maximization (EM),
quickly reaches material distributions more or less consistent with the measurements.
Note that we will also provide two alternative schemes in \cref{sec:PDscheme,sec:ADMM} that may be more efficient in the regime of low regularization.

\subsection{Alternation and forward backward splitting}
In the $k$th iteration we first minimize for $y$ with $\boldsymbol w$ fixed, which by \cref{thm:optimalY} yields
\begin{equation}\label{eqn:stepY}
y^k
=\mathscr I(\boldsymbol w^{k-1})\exp\left(\tfrac f{\sigma^2}-W_0\left(\mathscr F(\boldsymbol w^{k-1})\exp\left(\tfrac{f}{\sigma^2}\right)\right)\right)
=\mathscr I(\boldsymbol w^{k-1})\tfrac{W_0\left(\mathscr F(\boldsymbol w^{k-1})\exp\left(\tfrac{f}{\sigma^2}\right)\right)}{\mathscr F(\boldsymbol w^{k-1})},
\end{equation}
where we used the identity $\exp(-W_0(z))=W_0(z)/z$.
Then we fix $y$ to the value $y^k$ and perform one minimization step for $\boldsymbol w$.
To this end we first calculate the first order necessary optimality condition for minimizing $F(y^k,\cdot)$.
The G\^ateaux derivative of the data term in direction $\boldsymbol\phi=(\phi_1,\ldots,\phi_{N_{\text{mat}}})$ reads
\begin{equation*}
\partial_{\boldsymbol w}\left(\int_\Sigma\int_{E_{\min}}^{E_{\max}}d_{\mathrm{KL}}(y^k,\mathscr I(\boldsymbol w))\,\mathrm d\mathcal{L}^1\,\mathrm d\mathcal H^n\right)(\boldsymbol\phi)
=\sum_{i=1}^{N_{\text{mat}}}\int_\Sigma\int_{E_{\min}}^{E_{\max}}(y^k-\mathscr I(\boldsymbol w))g_i\,\mathrm d\mathcal{L}^1\,\mathscr D\phi_i\,\mathrm d\mathcal H^n,
\end{equation*}
thus the optimality conditions can be computed as
\begin{equation}\label{eqn:ocW}
0\in
\mathscr D^*\left(\int_{E_{\min}}^{E_{\max}}(y^k-\mathscr I(\boldsymbol w))g_i\,\mathrm d\mathcal{L}^1\right)
+\alpha\partial R_1(\boldsymbol w)_i-\beta(w_i-1/N_{\text{mat}}),\quad
i=1,\ldots,N_{\text{mat}},
\end{equation}
with $\partial R_1(\boldsymbol w)_i$ denoting the $i$th component of the convex subdifferential of $R_1$.
Multiplying by $(w_i+\varepsilon)/\mathscr{D}^*(\int_{E_{\min}}^{E_{\max}}y^kg_i\,\mathrm{d}\mathcal{L}^1)$ for some $\varepsilon>0$ yields
\begin{equation*}
0\in
w_i-w_{i,\text{EM}}^\varepsilon(\boldsymbol w)
+(w_i+\varepsilon)\frac{\alpha\partial R_1(\boldsymbol w)_i-\beta(w_i-1/N_{\text{mat}})}{\mathscr D^*\left(\int_{E_{\min}}^{E_{\max}}y^kg_i\,\mathrm d\mathcal{L}^1\right)},\quad
i=1,\ldots,N_{\text{mat}},
\end{equation*}
where we abbreviated the expectation maximization (EM) step of the data term (which is typically performed in an EM algorithm) as
\begin{equation}\label{eqn:stepWEM}
w_{i,\text{EM}}^\varepsilon(\boldsymbol w)
=(w_i+\varepsilon)\frac{\mathscr D^*\left(\int_{E_{\min}}^{E_{\max}}\mathscr I(\boldsymbol w)g_i\,\mathrm d\mathcal{L}^1\right)}{\mathscr D^*\left(\int_{E_{\min}}^{E_{\max}}y^kg_i\,\mathrm d\mathcal{L}^1\right)}-\varepsilon
\end{equation}
We here have an $\varepsilon$ of order $1$ in mind.
Taking a positive $\varepsilon$ is crucial since we actually aim for (and encourage by our model) characteristic functions $w_i$.
For zero $\varepsilon$, multiplying with $w_i+\varepsilon$ would thus lead to large parts of the optimality conditions being annihilated,
and in the subsequently described algorithm this would lead to zero function values staying zero.
We next view the optimality conditions as a fixed point equation for $\boldsymbol w$ and turn it into a fixed point iteration:
To this end we introduce a damping parameter $\omega_k\in(0,1)$ (a judicious choice of which will ensure a monotone energy decrease)
and let $\boldsymbol w^k$ solve
\begin{equation}\label{eqn:oc}
0\in
\frac{w_i^k-(1-\omega_k)w_i^{k-1}}{\omega_k}-w_{i,\text{EM}}^\varepsilon(\boldsymbol w^{k-1})
+(w_i^{k-1}+\varepsilon)\frac{\alpha\partial R_1(\boldsymbol w^k)_i-\beta(w_i^{k-1}-1/N_{\text{mat}})}{\mathscr D^*\left(\int_{E_{\min}}^{E_{\max}}y^kg_i\,\mathrm d\mathcal{L}^1\right)}
\end{equation}
for $i=1,\ldots,N_{\text{mat}}$.
This, however, is nothing else but the first order optimality condition of the convex minimization problem
\begin{align}\label{eqn:stepW}
\boldsymbol w^k
&=\argmin_{\boldsymbol w}
\frac12\int_\Omega\sum_{i=1}^{N_{\text{mat}}}
r_i^k\left(w_i-v_i^k\right)^2
\,\mathrm d\mathcal L^n
+R_1(\boldsymbol w)\\
&\nonumber\text{with }r_i^k=\frac{\mathscr D^*\left(\int_{E_{\min}}^{E_{\max}}y^kg_i\,\mathrm d\mathcal{L}^1\right)}{\omega_k\alpha(w_i^{k-1}+\varepsilon)}\\
&\nonumber\text{and }v_i^k=(1-\omega_k)w_i^{k-1}+\omega_k\left(w_{i,\text{EM}}^\varepsilon(\boldsymbol w^{k-1})+\frac{\beta(w_i^{k-1}+\varepsilon)(w_i^{k-1}-1/N_{\text{mat}})}{\mathscr D^*\left(\int_{E_{\min}}^{E_{\max}}y^kg_i\,\mathrm d\mathcal{L}^1\right)}\right).
\end{align}
It turns out that $\omega_k$ may typically be chosen at the order of $0.1$, and the proposed algorithm works efficiently in case of noisy data with high regularization.
However, in case of high signal-to-noise ratio (which typically means large $y^k$) and consequently small $\alpha$,
the weights $r_i^k$ become so big that the change of $\boldsymbol w$ in each step of the iteration becomes neglegibly small.
In this setting the algorithm requires an infeasibly large number of iterations so that we will propose an alternative algorithm for that situation.

Summarizing, we propose the three-step iterative scheme given in \cref{alg:modifiedEM}, details of which are provided in the subsequent paragraphs.

\makeatletter
\def\plist@algorithm{Algorithm\space}
\makeatother
\begin{algorithm}[h]
	\caption{Iterative algorithm for polyenergetic CT reconstruction}
	\label{alg:modifiedEM}
	\begin{algorithmic}
		\For{$k=1,\ldots$}
		\State calculate $y^k$ via \eqref{eqn:stepY}
			\State calculate $w_{i,\text{EM}}(\boldsymbol w^{k-1})$ via \eqref{eqn:stepWEM} for $i=1,\ldots,N_{\text{mat}}$
		\State calculate $\boldsymbol w^k$ via \eqref{eqn:stepW}
		\EndFor
	\end{algorithmic}
\end{algorithm}

\subsection{Approximating the Lambert-W function}
The computation of \eqref{eqn:stepY} requires the numerical evaluation of an expression of the form $W_0(a\exp b)/a$ for $W_0$ the Lambert-W function.
We simply achieve this via Newton's method as follows.
The relation $z = W_0(a\exp b)/a$ is by definition of the Lambert-W function equivalent to $az\exp(az) = a\exp b$, which can be rewritten as
\begin{equation*}
0=f(z)
\quad\text{for }
f(z)=\tfrac{\log z}a+z-\tfrac ba.
\end{equation*}
The Newton iteration with iterates $z^j$ to find the zero of $f$ thus reads
\begin{equation*}
z^{j+1} = z^j-\tfrac{f(z^j)}{f'(z^j)} = z^j-\tfrac{(\log z^j)/a+z^j-b/a}{1+1/(az^j)} = \tfrac{1-\log z^j+b}{a+1/z^j}.
\end{equation*}
We simply perform a small fixed number of iterations, starting from the initialization
\begin{equation*}
z^0=\begin{cases}
\exp b&\text{if }a\exp b<\mathrm e,\\
\log a+b-\log(\log a+b)(1-\tfrac1{2\log a+2b})&\text{else,}
\end{cases}
\end{equation*}
which was chosen due to $W_0(s)\approx s$ for small $s$ and $W_0(s)\approx \log(s)-\log(\log s)(1-\frac1{2\log s})$ for large $s$.

\subsection{Primal-dual iteration for weighted ROF model}\label{sec:ROF}
Optimization problem \eqref{eqn:stepW} is a version of the well-known Rudin--Osher--Fatemi (ROF) model \cite{rudin1992ROF}.
The difference lies in the weighted $L^2$-norm and the additional probability simplex constraint.
Various standard convex optimization algorithms can be applied to this problem,
for instance the (alternating) split Bregman iteration \cite{golstein2009}
or the equivalent alternating direction method of multipliers (ADMM) (for the equivalence see for instance \cite{TaWu09}).
In our experiments the primal-dual iteration by Chambolle and Pock \cite{Chambolle2011} converged a little faster for \eqref{eqn:stepW},
which is why we briefly summarize it here.

The saddle-point formulation of \eqref{eqn:stepW} requires a dual variable $\boldsymbol p=(p_1,\ldots,p_{N_{\text{mat}}})$ and reads
\begin{align*}\label{eqn:PDWROF}
\min_{\boldsymbol{w}}\max_{\boldsymbol{p}}\,
&G_1(\boldsymbol{w})-G_2(\boldsymbol{p})+\sum_{i=1}^{N_{\text{mat}}}\langle\nabla w_i,p_i\rangle
\qquad\text{with}\\
&G_1(\boldsymbol{w})
=\frac{1}{2}\int_{\Omega}\sum_{i=1}^{N_{\text{mat}}}r^k_i\left(w_i-v_i^k\right)^2\,\mathrm{d}\mathcal{L}^n
+\int_\Omega\iota_\Delta(\boldsymbol{w})\,\mathrm{d}\mathcal{L}^n,\\
&G_2(\boldsymbol{p})
=\int_\Omega\sum_{i=1}^{N_{\text{mat}}}\iota_{B_1}(p_i)\,\mathrm{d}\mathcal{L}^n,
\end{align*}
where as before $\iota_A$ denotes the convex indicator function of a set $A$ and where $B_1\subset\R^n$ is the Euclidean unit ball.
The corresponding primal-dual method with primal and dual step sizes $\tau,\sigma>0$ and relaxation parameter $\theta\in[1,2]$ then reads
\begin{align*}
\boldsymbol p^{k+1}&={\prox}_{\sigma G_2}(\boldsymbol p^k+\sigma\nabla\overline{\boldsymbol w}^k),\\
\boldsymbol w^{k+1}&={\prox}_{\tau G_1}(\boldsymbol w^k+\tau\div\boldsymbol p^{k+1}),\\
\overline{\boldsymbol w}^{k+1}&=\boldsymbol w^k+\theta(\boldsymbol w^{k+1}-\boldsymbol w^k).
\end{align*}
Abbreviating $\tilde p_i^k=p_i^k+\sigma\nabla\overline w_i^{k}$, the first step just becomes the pointwise projection
\begin{equation*}
p_i^{k+1}(x)=\tilde{p}^k_i(x)/\max\{1,|\tilde{p}^k_i(x)|\}\qquad\text{for }i=1,\ldots,N_{\text{mat}}.
\end{equation*}
Similarly, abbreviating $\tilde w_i^k=w_i^k+\tau\div p_i^{k+1}$, the second step can be written as
\begin{align}
\boldsymbol w^{k+1}(x)
&=\argmin_{\boldsymbol w}\frac12\sum_{i=1}^{N_{\text{mat}}}|w_i-\tilde w_i^k(x)|^2
+\tau\left(\frac{1}{2}\int_{\Omega}\sum_{i=1}^{N_{\text{mat}}}r^k_i(x)\left(w_i-v_i^k(x)\right)^2\right)
+\tau\iota_\Delta(\boldsymbol{w})\notag\\
&=\argmin_{\boldsymbol w}\frac12\sum_{i=1}^{N_{\text{mat}}}(1+\tau r_i^k(x))\left|w_i-\frac{\tilde w_i^k(x)+\tau r_i^k(x)v_i^k(x)}{1+\tau r_i^k(x)}\right|^2+\iota_\Delta(\boldsymbol{w}),
\label{eqn:PD_ROF_wUpdate}
\end{align}
whose solution is described in \cref{sec:simplexProjection}.

%

\subsection{Weighted projection onto probability simplex}\label{sec:simplexProjection}
The numerical solution of \eqref{eqn:PD_ROF_wUpdate} requires the computation of the orthogonal projection onto the probability simplex $\Delta$ with respect to a weighted Euclidean inner product.
More specifically, given positive weights $r_1,\ldots,r_m$ and values $v_1,\ldots,v_m$ (where in our setting $m=N_{\text{mat}}$) we need to solve the minimization problem
\begin{equation}\label{eq:projwsimplex}
 \min_{\boldsymbol{w}\in\Delta}\frac{1}{2}\sum_{i=1}^{m}r_i(w_i-v_i)^2.
\end{equation}
Our solution to the problem is a straightforward generalization of \cite{chen2011projection}, which solves the problem for the special case $r_1=\ldots=r_m=1$.

We begin by stating a small generalization of Moreau's identity.
\begin{lemma}[Moreau's identity]\label{thm:Moreau}
	Let  $A\in\R^{m\times m}$ be positive semi-definite and $f:\R^m\to(-\infty,\infty]$ be proper and convex with Legendre--Fenchel conjugate $f^*$, then for all $\boldsymbol v\in\R^m$ we have
	   \begin{equation*}\label{eq:projwsimplex1}
	   \boldsymbol{v}= \argmin_{\boldsymbol{w}}\left\{\tfrac{1}{2}(\boldsymbol{w}\!-\!\boldsymbol{v})^TA(\boldsymbol{w}\!-\!\boldsymbol{v}) +f(\boldsymbol{w})\right\} + A^{\!-\!1}\argmin_{\boldsymbol{w}}\left\{\tfrac{1}{2}(\boldsymbol{w}\!-\!A\boldsymbol{v})^TA^{\!-\!1}(\boldsymbol{w}\!-\!A\boldsymbol{v})+f^*(\boldsymbol{w})\right\}.
	   \end{equation*}
\end{lemma}
\begin{proof}
Introducing $\boldsymbol{z}=\sqrt{A}\boldsymbol{w}$ and $g(\boldsymbol{z}) = f\left(\sqrt{A^{-1}}\boldsymbol{z}\right)$, the first term on the right-hand side can be rewritten as
\begin{equation*}
R
=\argmin_{\boldsymbol{w}}\left\{\tfrac{1}{2}(\boldsymbol{w}\!-\!\boldsymbol{v})^TA(\boldsymbol{w}\!-\!\boldsymbol{v}) +f(\boldsymbol{w})\right\}
=\sqrt{A^{-1}}\argmin_{\boldsymbol{z}}\left\{\tfrac{1}{2}|\boldsymbol{z}-\sqrt A\boldsymbol{v}|^2+g(\boldsymbol{z})\right\},
\end{equation*}
while the substitution $\boldsymbol{z}=\sqrt{A^{-1}}\boldsymbol{w}$ turns the second term into
\begin{equation*}
S
=A^{\!-\!1}\argmin_{\boldsymbol{w}}\left\{\tfrac{1}{2}(\boldsymbol{w}\!-\!A\boldsymbol{v})^TA^{\!-\!1}(\boldsymbol{w}\!-\!A\boldsymbol{v})+f^*(\boldsymbol{w})\right\}
=\sqrt{A^{\!-\!1}}\argmin_{\boldsymbol{z}}\left\{\tfrac{1}{2}|\boldsymbol{z}\!-\!\sqrt A\boldsymbol{v}|^2+f^*\left(\sqrt{A}\boldsymbol{z}\right)\right\}.
\end{equation*}
Note now that the substitution $\boldsymbol{w}=\sqrt{A}\boldsymbol{z}$ implies
\[
g^*(\boldsymbol{a})
= \sup_{\boldsymbol{z}} \:\boldsymbol{z}^T\boldsymbol{a} -f\left(\sqrt{A^{-1}}\boldsymbol{z}\right) 
=\sup_{\boldsymbol{w}}\:\boldsymbol{w}^T\sqrt{A}\boldsymbol{a} -f^*\left(\boldsymbol{w}\right) 
=f^{*}(\sqrt{A}\boldsymbol{a}).
\]
Thus we have
\[
R+S = \sqrt{A^{-1}}\left(\prox\hfil_{g}(\sqrt{A}\boldsymbol{v})+\prox\hfil_{g^*}(\sqrt{A}\boldsymbol{v})\right)
=\sqrt{A^{-1}}\sqrt{A}\boldsymbol{v}
= \boldsymbol{v}.
\qedhere
\]
\end{proof}

We next state the solution to \eqref{eq:projwsimplex}.

\begin{prop}[Weighted projection onto probability simplex]\label{thrm:projontosimplexsolution}
	Let $j_1,\ldots,j_m$ be a permutation such that $r_{j_{1}}v_{j_{1}}\leq \ldots \leq r_{j_{m}}v_{j_{m}}$, and define
	\begin{equation}\label{eq:ti}
	 t_i = \frac{\sum_{k=i+1}^{m}v_{j_k}-1}{\sum_{k=i+1}^{m}\frac{1}{r_{j_k}}}.
	\end{equation}
	Then the solution of (\ref{eq:projwsimplex}) is given by 
	\[
	  w_i = \max\left(0, v_i - \tfrac{\tilde{t}}{r_i} \right),\quad i=1,\ldots m,
	\]
	where $\tilde{t}=t_k$ for that index $k\in\{1,\ldots,m-1\}$ which satisfies $r_{j_{k}}v_{j_{k}}\leq t_k \leq r_{j_{k+1}}v_{j_{k+1}}$.
\end{prop}
\begin{proof}
Applying \cref{thm:Moreau} with $f=\iota_\Delta$ the indicator function of $\Delta$ and $A=\mathrm{diag}(r_1,\ldots,r_m)$ a diagonal matrix, we obtain
	\begin{equation*}\label{eq:proxdeltaq1}
	\argmin_{\boldsymbol{w}\in\Delta}\frac{1}{2}\sum_{i=1}^{m}r_i(w_i-v_i)^2
	= \boldsymbol{v}- A^{-1}\argmin_{\boldsymbol{w}}\left\{\frac{1}{2}\sum_{i=1}^m \frac{1}{r_i} (w_i-r_iv_i)^2 + i^*_{\Delta}(\boldsymbol{w})\right\}.
	\end{equation*}
	Now we observe that 
	\begin{equation*}\label{eq:dual indicator function}
		i^*_{\Delta}(\boldsymbol{w}) = \sup_{\boldsymbol{v}\in\Delta}\: \boldsymbol{w}^T\boldsymbol{v} = \max(w_1,\ldots,w_n),
	\end{equation*}
	and that 
	\[\min_{\boldsymbol{w}}\bigg\{\frac{1}{2}\sum_{i=1}^m \frac{1}{r_i} (w_i-r_iv_i)^2 + i^*_{\Delta}(\boldsymbol{w})\bigg\} = \min_t f(t)
	\]
	for
	\begin{equation*}\label{eq:funct}
		f(t)  = \min\left\{t+\frac{1}{2}\sum_{i=1}^m \frac{1}{r_i} (w_i-r_iv_i)^2\ \middle|\ \boldsymbol w\in\R^m,\ \max(w_1,\ldots,w_n)=t \right\}.
	\end{equation*}
For fixed $t$, the minimizer in the definition of $f$ is 
\[
  w_i(t) = \begin{cases}
  t & \text{if } r_iv_i>t, \\
  r_iv_i & \text{else},
  \end{cases}
\]
thus we can write 
\[
f(t) = t +\frac{1}{2}\sum_{r_iv_i\leq t}  \frac{1}{r_i}(t - r_iv_i)^2.
\] 
 Note that $f$ is piecewise quadratic, convex, continuous and has continuous derivative
 \[
  f'(t) = 1 +\sum_{r_iv_i\leq t} \left(\tfrac{t}{r_i} - v_i\right).
 \]
 Its unique minimizer satisfies $f'(t)=0$ and is thus given by $t=t_k$ for that index $k$ which satisfies $r_{j_k}v_{j_k}\leq  t_k \leq r_{j_{k+1}}v_{j_{k+1}}$.
 The solution to the original optimization problem is therefore given by $\boldsymbol w=\boldsymbol v-A^{-1}\boldsymbol w(\tilde t)$, which yields the desired formula.
\end{proof}

\subsection{Reinitialization}\label{subsec:reint}
Being a nonconvex optimization problem (in particular due to the regularization $R_2$ which promotes pure materials),
one cannot avoid running into local minima in which a pure material at a point $x$ is replaced with a convex combination of other materials.
To turn such a solution into one of pure materials we employ the following simple procedure.
At each point $x\in\Omega$ we calculate the effective attenuation coefficient
\begin{equation*}
g_x=\sum_{i=1}^{N_{\text{mat}}}w_i(x)g_i:[E_{\min},E_{\max}]\to[0,\infty)
\end{equation*}
and identify that material $j$ whose attenuation coefficient $g_i$ is closest to $g_x$ in an $L^2$-sense,
$j=\argmin_i\|g_x-g_i\|_{L^2}$.
We then simply reinitialize $w_j(x)=1$ and $w_i(x)=0$ for $i\neq j$.

\subsection{Convergence analysis}
In the following, we abbreviate by $\{a_k\}$ the sequence $a_k$ for $k\in\N$.
Before proving convergence of the algorithm, we show its monotonicity and provide an estimate for a feasible damping parameter.

\begin{prop}[Monotone descent of $F$] \label{thm:monotoneDescent}
Let $\{y^k,\boldsymbol{w}^k\}$ be the sequence produced by \cref{alg:modifiedEM}
for $\varepsilon>0$ and damping parameters $\{\omega_k\}$, and abbreviate
\begin{align*}
G
&=\int_{E_{\min}}^{E_{\max}}I_0(E)\boldsymbol{g}\otimes \boldsymbol{g}\,\mathrm d\mathcal{L}^1
\in\R^{N_{\text{mat}}\times N_{\text{mat}}},\\
A^{\epsilon}_k
&=\int_{\Omega}\sum_{i=1}^{N_{\text{mat}}} \frac{\mathscr{D}^*\left(\int_{E_{\min}}^{E_{\max}} y^{k}g_i\mathrm d\mathcal{L}^1\right) (w_i^{k}-w_i^{k-1})^2}{w_i^{k-1}+\varepsilon }\mathrm{d}\mathcal{L}^n
\in\R,\\
C_k
&=\frac{1}{2}\int_{\Sigma}(\mathscr{D}(\boldsymbol{w}^{k}-\boldsymbol{w}^{k-1}))^T G\mathscr{D}(\boldsymbol{w}^{k}-\boldsymbol{w}^{k-1})\mathrm{d}\mathcal{H}^n
\in\R.
\end{align*}
Further, let $\eta\in(0,1)$ be arbitrary. If in each iteration $k$ we have
\begin{equation*}
\label{eq:3.3.2}
\omega_{k}\leq \frac{A^{\varepsilon}_k}{C_k}(1-\eta),
\end{equation*}
then $F(y^k,\boldsymbol{w}^k)$ is monotonically decreasing in $k$ with
\[
F(y^k,\boldsymbol{w}^k)+\frac\eta{\omega_k}A_k^\varepsilon\leq F(y^k,\boldsymbol{w}^{k-1})\leq F(y^{k-1},\boldsymbol{w}^{k-1}).
\]
\end{prop}

Before proving the result, let us state a few properties of the involved quantities.

\begin{lemma}[Positivity of $G$]
If the attenuation coefficients $g_i:[E_{\min},E_{\max}]\to[0,\infty)$ are linearly independent (as Lebesgue functions) and $I_0$ is positive,
the matrix $G$ from \cref{thm:monotoneDescent} is positive definite.
\end{lemma}
\begin{proof}
By definition $G$ is positive semi-definite.
Now assume $\boldsymbol w^TG\boldsymbol w=0$ for some $\boldsymbol w\in\R^{N_{\text{mat}}}$,
then this implies $\sum_{i=1}^{N_{\text{mat}}}w_ig_i=0$ almost everywhere, contradicting the linear independence.
\end{proof}

Linear independence of the $g_i$ in essence means that the materials could be reconstructed if the X-ray attenuation were measured separately for all photon energies.
The resulting positive definiteness of $G$ then implies that $C_k>0$ unless the method has converged, i.\,e.\ $\boldsymbol w^{k-1}=\boldsymbol w^k$;
thus the expression $\frac{A_k^\varepsilon}{C_k}$ makes sense.

\begin{lemma}[Positivity of $y^k$]\label{thm:positivityY}
If the attenuation coefficients $g_i$ are uniformly bounded and $I_0$ positive, then $y^k$ is strictly positive.
If in addition the measurement $f$ is bounded below and $I_0$ is bounded away from zero,
then also $y^k$ is bounded away from zero by a constant depending on $f$, $I_0$, and $\boldsymbol g$.
\end{lemma}
\begin{proof}
This is a straightforward consequence of the definition of $y^k$ via \eqref{eqn:stepY}
and the positivity of $\mathscr{I}(\boldsymbol w^{k-1})$ as well as $\mathscr{F}(\boldsymbol w^{k-1})$.
\end{proof}

If $y^k$ is strictly positive, then so is $A_k^\varepsilon$ (again unless the method has already converged) so that a feasible choice of the damping parameter $\omega_k$ exists.

\begin{lemma}[Lower bound on damping]\label{thm:dampingBound}
On $L^2(\Omega;\R^{N_{\text{mat}}})$ and for positive $y^k\in L^\infty((E_{\min},E_{\max})\times\Sigma)$ consider the symmetric quadratic form
\begin{equation*}
Q_{y^k}(\boldsymbol\psi,\boldsymbol\psi)=
\frac{1}{2}\int_{\Sigma}\left(\mathscr{D}\left(\psi_i/s_i^k\right)\right)^T G\mathscr{D}\left(\psi_i/s_i^k\right)\,\mathrm{d}\mathcal{H}^n
\quad\text{with }
s_i^k=\sqrt{\mathscr{D}^*\left(\int_{E_{\min}}^{E_{\max}} y^{k}g_i\mathrm d\mathcal{L}^1\right)},
\end{equation*}
then $\frac{A_k^\varepsilon}{C_k}\geq\frac1{1+\varepsilon}\frac1{\lambda_{y^k}}$
for $\lambda_{y^k}=\sup_{\|\boldsymbol\psi\|_{L^2}\leq1}Q_{y^k}(\boldsymbol\psi,\boldsymbol\psi)$ the norm or maximum eigenvalue of the operator associated with $Q_{y^k}$.
\end{lemma}
\begin{proof}
Abbreviating $\psi_i=s_i^k(w_i^k-w_i^{k-1})$ we have
\begin{equation*}
\frac{A_k^\varepsilon}{C_k}
\geq\frac{\int_{\Omega}\sum_{i=1}^{N_{\text{mat}}} \frac{\mathscr{D}^*\left(\int_{E_{\min}}^{E_{\max}} y^{k}g_i\mathrm d\mathcal{L}^1\right) (w_i^{k}-w_i^{k-1})^2}{1+\varepsilon}\,\mathrm{d}\mathcal{L}^n}{\frac{1}{2}\int_{\Sigma}(\mathscr{D}(\boldsymbol{w}^{k}-\boldsymbol{w}^{k-1}))^T G\mathscr{D}(\boldsymbol{w}^{k}-\boldsymbol{w}^{k-1})\,\mathrm{d}\mathcal{H}^n}
=\frac1{1+\varepsilon}\frac{\|\boldsymbol\psi\|_{L^2}^2}{Q_{y^k}(\boldsymbol\psi,\boldsymbol\psi)}
\qedhere
\end{equation*}
\end{proof}

\begin{remark}[Choice of damping parameter]\label{rem:damping}
By \cref{thm:positivityY,thm:dampingBound} the damping parameter $\omega_k$ can be chosen strictly positive.
Moreover, replacing $y^k$ in \cref{thm:dampingBound} with its lower bound $\underline y$ from \cref{thm:positivityY},
we obtain a feasible $\omega_{k}=\frac1{1+\varepsilon}\frac{1-\eta}{\lambda_{\underline y}}>0$ which can be computed beforehand.
In our numerical computations we employed this feasible choice, and its value was typically around $0.1$.
\end{remark}

Finally, before we prove \cref{thm:monotoneDescent} let us rewrite the optimality condition \eqref{eqn:oc}.
By inserting the definition of $w^{\varepsilon}_{i,\text{EM}}(\boldsymbol{w}^{k-1}))$
and multiplying with $\mathscr{D}^*\left(\int_{E_{\min}}^{E_{\max}}y^kg_i\,\mathrm d\mathcal{L}^1\right)/(w^{k-1}_i+\varepsilon)$
it turns into
\begin{equation}\label{eqn:ocSimplified}
0=\frac{\mathscr{D}^*\left(\int_{E_{\min}}^{E_{\max}}y^kg_i\,\mathrm d\mathcal{L}^1\right)(w_i^{k}-w_i^{k-1})}{\omega_k(w^{k-1}_i+\varepsilon)}
-\mathscr D^*\left(\int_{E_{\min}}^{E_{\max}}\mathscr I(\boldsymbol w^{k-1})g_i-y^kg_i\,\mathrm d\mathcal{L}^1\right)
+\alpha p_i^k+\beta R'_2(\boldsymbol{w}^{k-1})_i
\end{equation}
for $i=1,\ldots,N_{\text{mat}}$,
where $R'_2$ denotes the Fr\'echet derivative of $R_2$ and $(p^k_1,\ldots,p^k_{N_{\text{mat}}})$ the element of the subdifferential $\partial R_1(\boldsymbol w^k)$
such that the optimality condition is satisfied.

\begin{proof}[Proof of \cref{thm:monotoneDescent}]
By definition of $y^k=\argmin F(\cdot,\boldsymbol w^{k-1})$ we have $F(y^k,\boldsymbol w^{k-1})\leq F(y^{k-1},\boldsymbol w^{k-1})$,
so it remains to prove $F(y^k,\boldsymbol w^{k})\leq F(y^{k},\boldsymbol w^{k-1})$.
Multiplying \eqref{eqn:ocSimplified} with $w_i^{k}-w_i^{k-1}$,
summing over $i$ and integrating we obtain
\begin{multline*}
0
=\frac{A_k^\varepsilon}{\omega_k}
-\int_{\Omega}\sum_{i=1}^{N_{\text{mat}}}\mathscr D^*\left(\int_{E_{\min}}^{E_{\max}}\mathscr I(\boldsymbol w^{k-1})g_i-y^kg_i\,\mathrm d\mathcal{L}^1\right)(w_i^{k}-w_i^{k-1})\,\mathrm{d}\mathcal{L}^n\\
+\alpha\sum_{i=1}^{N_{\text{mat}}}\left\langle\   p^k_i,w_i^{k}-w_i^{k-1}\right\rangle
+\beta\left\langle R'_2(\boldsymbol{w}^{k-1}),\boldsymbol w^{k}-\boldsymbol w^{k-1}\right\rangle.
\end{multline*}
By definition of the subgradient and by the convexity of $-R_2$ we have
\begin{align*}
\sum_{i=1}^{N_{\text{mat}}}\left\langle p^k_i,w_i^{k}-w_i^{k-1}\right\rangle
&\geq R_1(\boldsymbol{w}^{k})-R_1(\boldsymbol{w}^{k-1}),\\
\left\langle R'_2(\boldsymbol{w}^{k-1})_i,\boldsymbol w^{k}-\boldsymbol w^{k-1}\right\rangle
&\geq R_2(\boldsymbol{w}^{k})-R_2(\boldsymbol{w}^{k-1}).
\end{align*}
so that the previous equality can be turned into the inequality
\begin{align*}
\alpha R_1(\boldsymbol{w}^{k})+\beta& R_2(\boldsymbol{w}^{k})- \alpha R_1(\boldsymbol{w}^{k-1})-\beta R_2(\boldsymbol{w}^{k-1})
+\frac{A_k^\varepsilon}{\omega_{k}}\\
&\leq\int_{\Sigma}\sum_{i=1}^{N_{\text{mat}}}\left(\int_{E_{\min}}^{E_{\max}}\mathscr I(\boldsymbol w^{k-1})g_i-y^kg_i\,\mathrm d\mathcal{L}^1\right)\mathscr D(w_i^{k}-w_i^{k-1})\,\mathrm{d}\mathcal{H}^n.
\end{align*}
Abbreviating the Kullback--Leibler fidelity term as
\begin{equation*}
L(\boldsymbol{w})=\int_{\Sigma}\int_{E_{\min}}^{E_{\max}}d_{KL}(y^{k},\mathscr{I}(\boldsymbol{w}))\,\mathrm{d}\mathcal{L}^1\,\mathrm{d}\mathcal{H}^n,
\end{equation*}
we now add the difference $L(\boldsymbol{w}^k)-L(\boldsymbol{w}^{k-1})$ on both sides in order to achieve a term $F(y^k,\boldsymbol w^k)-F(y^k,\boldsymbol w^{k-1})$,
\begin{align*}
&F(y^k,\boldsymbol{w}^{k})-F(y^k,\boldsymbol{w}^{k-1})+\frac{A^\varepsilon_k}{\omega_k}\\
&\leq-\sum_{i=1}^{N_{\text{mat}}}\int_{\Sigma}\bigg[\int_{E_{\min}}^{E_{\max}}y^{k}g_i - \mathscr I(\boldsymbol w^{k-1})g_i\,\mathrm{d}\mathcal{L}^1\bigg]\mathscr D(w_i^{k}-w_i^{k-1})\,\mathrm{d}\mathcal{H}^n
+L(\boldsymbol{w}^k)-L(\boldsymbol{w}^{k-1})\\
&=L(\boldsymbol{w}^k)-L(\boldsymbol{w}^{k-1})-L'(\boldsymbol{w}^{k-1})(\boldsymbol{w}^{k}-\boldsymbol{w}^{k-1}),
\end{align*}
where $L'$ denotes the Fr\'echet derivative of $L$.
Using Taylor's theorem, the right-hand side of the inequality can be expressed as a second directional derivative of $L$,
\begin{multline*}
L(\boldsymbol{w}^k)-L(\boldsymbol{w}^{k-1})-L'(\boldsymbol{w}^{k-1})(\boldsymbol{w}^{k}-\boldsymbol{w}^{k-1})
=\frac12\frac{\mathrm{d}^2}{\mathrm{d}t^2}L(\boldsymbol{w}^{k-1}+t(\boldsymbol{w}^{k}-\boldsymbol{w}^{k-1}))|_{t=\tau}\\
=\frac12\int_{\Sigma}\sum_{i,k=1}^{N_{\text{mat}}}\int_{E_{\min}}^{E_{\max}}I_0\exp\left(-\boldsymbol{g}\cdot\mathscr{D}\boldsymbol{z}\right)\mathscr{D}\phi_i g_i \mathscr{D}\phi_kg_k\,\mathrm{d}\mathcal{L}^1\,\mathrm{d}\mathcal{H}^n
\leq\frac12\int_{\Sigma}(\mathscr{D}\boldsymbol{\phi})^TG\mathscr{D}\boldsymbol{\phi}\,\mathrm{d}\mathcal{H}^n,
\end{multline*}
where we abbreviated $\boldsymbol\phi=(\boldsymbol{w}^{k}-\boldsymbol{w}^{k-1})$ and $\boldsymbol z=\boldsymbol{w}^{k-1}+\tau(\boldsymbol{w}^{k}-\boldsymbol{w}^{k-1})$ for some $\tau\in[0,1]$.
Summarizing, we arrive at
\begin{equation*}
\label{eq:3.4.1n}
F(y^k,\boldsymbol{w}^{k})-F(y^k,\boldsymbol{w}^{k-1})+\frac{A_k^{\varepsilon}}{\omega_{k}}  \leq C_k
\end{equation*}
or equivalently
\begin{equation*}
\label{eq:3.4.2n}
F(y^k,\boldsymbol{w}^{k})+\frac{\eta}{\omega_{k}} A_k^{\varepsilon} +\frac{1-\eta}{\omega_{k}}A_k^{\varepsilon} \leq C_k +F(y^k,\boldsymbol{w}^{k-1}),
\end{equation*}
which implies $F(y^k,\boldsymbol{w}^{k})\leq F(y^k,\boldsymbol{w}^{k-1})$ if $\frac{1-\eta}{\omega_{k}}A_k^{\varepsilon} \geq C_k$.
\end{proof}

Next we prove convergence of \cref{alg:modifiedEM} to a critical point of the functional $F$.
We follow essentially the same steps as in \cite{brune2010,sawatzky2011,Gpap2020}.
%
%
%

\begin{theorem}[Convergence of damped alternating minimization]\label{Thrm: Convergence}
Assume $I_0$ to be integrable and $\boldsymbol g$ and $f$ to be bounded.
Also assume $I_0$ and $\boldsymbol g$ to be bounded away from zero.
Let $\{y^k,\boldsymbol{w}^k\}$ be the sequence produced by \cref{alg:modifiedEM}
for $\varepsilon>0$ and damping parameters $\{\omega_k\}$ satisfying the condition from \cref{thm:monotoneDescent} and being uniformly bounded away from zero
(which is possible by \cref{rem:damping}).
Then every subsequence of $\{y^k,\boldsymbol{w}^k\}$ contains a subsequence
converging strongly in $L^1((E_{\min},E_{\max})\times\Sigma)\times L^q(\Omega)^{N_{\text{mat}}}$ for any $q\in[1,\infty)$,
and every limit point of the sequence is a critical point of $F$.
\end{theorem}
\begin{proof}
We proceed in three steps.
\begin{enumerate}
\item
From coercivity of $F$ and \cref{thm:monotoneDescent} we derive convergence of $\boldsymbol w^k$ and $y^k$ along a subsequence.
\item
We exploit the optimality condition to derive convergence of the dual variables $p_i^k$.
\item
We show that the limit satisfies the optimality conditions for a critical point of $F$.
\end{enumerate}

\noindent\emph{Step 1 -- convergence of the primal iterates.}
\Cref{thm:monotoneDescent} implies
\begin{equation*}
\sum_{i=1}^{N_{\text{mat}}}\mathrm{TV}(w_i^k)-\frac12\mathcal{L}^n(\Omega)(1-1/N_{\text{mat}})^2
\leq F(y^k,\boldsymbol w^k)
\leq F(y^0,\boldsymbol w^0)
<\infty
\end{equation*}
so that $\boldsymbol w^k$ is uniformly bounded in $[L^\infty(\Omega)\cap\mathrm{BV}(\Omega)]^{N_{\text{mat}}}$.
Thus, any subsequence of $\{\boldsymbol w^k\}$ contains another subsequence $\{\boldsymbol w^{k_n}\}$
converging weakly-* in $[L^\infty(\Omega)\cap\mathrm{BV}(\Omega)]^{N_{\text{mat}}}$ to some $\boldsymbol w$.
By the compact embedding $\mathrm{BV}(\Omega)\hookrightarrow L^1(\Omega)$ and H\"older's inequality
this implies $\boldsymbol w^{k_n}\to\boldsymbol w$ strongly in $L^q(\Omega)^{N_{\text{mat}}}$ as $n\to\infty$ for any $q\in[1,\infty)$.
In fact, we even have $\boldsymbol w^{k_n+K}\to\boldsymbol w$ for any fixed integer $K$.
Indeed, \cref{thm:monotoneDescent} yields
\begin{equation*}
F(y^k,\boldsymbol w^k)
\leq F(y^{k-1},\boldsymbol w^{k-1})-\eta A_k^\varepsilon
\leq\ldots
\leq F(y^0,\boldsymbol w^0)-\eta\sum_{i=1}^kA_i^\varepsilon
\end{equation*}
so that $\sum_{i=1}^\infty A_i^\varepsilon$ is bounded and thus $A_k^\varepsilon\to0$ as $k\to\infty$.
Now \cref{thm:positivityY} implies $\mathscr{D}^*\left(\int_{E_{\min}}^{E_{\max}} y^{k}g_i\mathrm d\mathcal{L}^1\right)/(w_i^{k-1}+\varepsilon)>c$
for some positive constant $c$ independent of $k$ so that
\begin{equation*}
\|\boldsymbol w^k-\boldsymbol w^{k-1}\|_{L^2}^2
\leq cA_k^\varepsilon
\to0.
\end{equation*}
Again by H\"older's inequality we obtain $\|\boldsymbol w^k-\boldsymbol w^{k-1}\|_{L^q}\to0$ as $k\to\infty$ for any $q\in[1,\infty)$
and therefore by induction $\boldsymbol w^{k_n+K}\to\boldsymbol w$ as $n\to\infty$.

Next consider the convergence of $y^k$.
By \cref{lem:L1Continuity of Cpoly} we have $\mathscr{F}(\boldsymbol w^{k_n-1})\to\mathscr{F}(\boldsymbol w)$ in $L^1(\Sigma)$
so that also $Y_{\boldsymbol w^{k_n-1}}\to Y_{\boldsymbol w}$ in $L^1(\Sigma)$ as $n\to\infty$ for the quantity from \cref{thm:optimalY}.
Now by \cref{thm:optimalY} we have
\begin{equation*}
y^{k_n}
=y_{\boldsymbol w^{k_n-1}}
=\mathscr{I}({\boldsymbol w^{k_n-1}})\exp\left(\tfrac{f-Y_{\boldsymbol w^{k_n-1}}}{\sigma^2}\right)
=I_0\exp\left(-\sum_{i=1}^{N_{\text{mat}}}g_i\mathscr Dw_i^{k_n-1}+\tfrac{f-Y_{\boldsymbol w^{k_n-1}}}{\sigma^2}\right).
\end{equation*}
Repeating the argument of \cref{lem:L1Continuity of Cpoly}
(only with $-\sum_{i=1}^{N_{\text{mat}}}g_i\mathscr Dw_i$ replaced by $-\sum_{i=1}^{N_{\text{mat}}}g_i\mathscr Dw_i+\tfrac{f-Y_{\boldsymbol w}}{\sigma^2}$)
we obtain
\begin{equation*}
y^{k_n}\to y=y_{\boldsymbol w}
\end{equation*}
strongly in $L^1((E_{\min},E_{\max})\times\Sigma)$ as $n\to\infty$.
Note that from the mere coercivity of $F$ we would only have obtained weak convergence of a subsequence of $y^{k_n}$ to $y$.

\noindent\emph{Step 2 -- convergence of dual iterates.}
We first prove that $\int_{E_{\min}}^{E_{\max}}\mathscr{I}(\boldsymbol w^{k_n-1})g_i\,\mathrm d\mathcal L^1$ converges in any $L^q(\Sigma)$ with $q\in[1,\infty)$.
Indeed, exploiting that the exponential function has unit Lipschitz constant on the negative real axis,
\begin{align*}
&\int_\Sigma\left(\int_{E_{\min}}^{E_{\max}}\mathscr{I}(\boldsymbol w^{k_n-1})g_i\,\mathrm d\mathcal L^1
-\int_{E_{\min}}^{E_{\max}}\mathscr{I}(\boldsymbol w)g_i\,\mathrm d\mathcal L^1\right)^q\mathrm d\mathcal H^n\\
&\leq\int_\Sigma\left(\int_{E_{\min}}^{E_{\max}}I_0(E)g_i(E)\left|\exp\left(-\sum_{j=1}^{N_{\text{mat}}}g_j\mathscr Dw_j^{k_n-1}\right)-\exp\left(-\sum_{j=1}^{N_{\text{mat}}}g_j\mathscr Dw_j\right)\right|\,\mathrm d\mathcal L^1\right)^q\mathrm d\mathcal H^n\\
&\leq\int_\Sigma\left(\int_{E_{\min}}^{E_{\max}}I_0(E)g_i(E)\sum_{j=1}^{N_{\text{mat}}}g_j\left|\mathscr D(w_j^{k_n-1}-w_j)\right|\,\mathrm d\mathcal L^1\right)^q\mathrm d\mathcal H^n\\
&\leq\left(\int_{E_{\min}}^{E_{\max}}I_0(E)g_i(E)\max_jg_j(E)\,\mathrm d\mathcal L^1\right)^q\int_\Sigma\left|\mathscr D(\boldsymbol w^{k_n-1}-\boldsymbol w)\right|^q\mathrm d\mathcal H^n,
\end{align*}
which converges to zero as $\mathscr D(w_i^{k_n-1}-w_i)\to0$ in $L^q(\Sigma)$ by the continuity of $\mathscr D$.
As a consequence, for $n\to\infty$ we also have the limit
\begin{equation*}
\int_{E_{\min}}^{E_{\max}}y^{k_n}g_i\,\mathrm d\mathcal L^1
=\int_{E_{\min}}^{E_{\max}}\mathscr{I}(\boldsymbol w^{k_n-1})g_i\,\mathrm d\mathcal L^1\exp\left(\tfrac{f-Y_{\boldsymbol w^{k_n-1}}}{\sigma^2}\right)
\to\int_{E_{\min}}^{E_{\max}}yg_i\,\mathrm d\mathcal L^1
\end{equation*}
strongly in any $L^q(\Sigma)$ with $q\in[1,\infty)$
(recall that $Y_{\boldsymbol w^{k_n-1}}$ converges in $L^1(\Sigma)$ and thus the exponential term converges in any $L^q(\Sigma)$).
Now we can rewrite \eqref{eqn:ocSimplified} as
\begin{equation*}
p_i^{k_n}=\frac{\mathscr{D}^*\left(\int_{E_{\min}}^{E_{\max}}y^{k_n}g_i\,\mathrm d\mathcal{L}^1\right)(w_i^{{k_n}}-w_i^{{k_n}-1})}{-\alpha\omega_{k_n}(w^{{k_n}-1}_i+\varepsilon)}
+\frac1\alpha\mathscr D^*\left(\int_{E_{\min}}^{E_{\max}}\mathscr I(\boldsymbol w^{{k_n}-1})g_i-y^{k_n}g_i\,\mathrm d\mathcal{L}^1\right)
+\frac\beta\alpha\left(w_i^{{k_n}-1}-\frac1{N_{\text{mat}}}\right),
\end{equation*}
which by the continuity of $\mathscr D^*$ converges strongly in any $L^q(\Omega)$ with $q\in[1,\infty)$ to
\begin{equation}\label{eqn:ocRearranged}
p_i=\frac1\alpha\mathscr D^*\left(\int_{E_{\min}}^{E_{\max}}\mathscr I(\boldsymbol w)g_i-yg_i\,\mathrm d\mathcal{L}^1\right)
+\frac\beta\alpha\left(w_i-\frac1{N_{\text{mat}}}\right).
\end{equation}
By the continuous embedding of $\mathrm{BV}(\Omega)$ into $L^{\frac n{n-1}}(\Omega)$ this also implies weak-* convergence of $p_i^{k_n}$ to $p_i$ in the dual space $\mathrm{BV}(\Omega)^*$.

\noindent\emph{Step 3 -- criticality of limit.}
We first show that $p_i\in\partial R_1(w_i)$.
To this end let $z\in\mathrm{BV}(\Omega)$ be arbitrary, then
\begin{equation*}
R_1(w_i^{k_n})+\langle p_i^{k_n},z-w_i^{k_n}\rangle\leq R_1(z)
\end{equation*}
since $p_i^{k_n}$ is a subgradient of $R_1$ at $w_i^{k_n}$.
Due to the sequential lower semi-continuity of $R_1$ and the strong convergence of both $p_i^{k_n}$ and $w_i^{k_n}$ in any Lebesgue space, taking the limit $n\to\infty$ on both sides yields
\begin{equation*}
R_1(w_i)+\langle p_i,z-w_i\rangle\leq R_1(z).
\end{equation*}
By the arbitrariness of $z$ we see $p_i\in\partial R_1(w_i)$.
Therefore, \eqref{eqn:ocRearranged} implies that the optimality condition \eqref{eqn:ocW} for minimizing $F$ with respect to $\boldsymbol w$ is satisfied.
The optimality condition for minimizing $F$ with respect to $y$ is also satisfied due to $y=y_{\boldsymbol w}$, hence the point $(y,\boldsymbol w)$ is critical for $F$.
\end{proof}

%
%

\section{Numerical results}\label{sec:Numresults}

As examples supporting the feasibility and usefulness of the proposed model we will reconstruct a phantom from measurements with different noise levels.

\subsection{Phantom data}

We use the phantom from \cite{Gao_2011} (\cref{fig:Phantom_attenuation} left) consisting of disks of different materials
(we will use up to six different materials including air).
The different materials also stem from \cite{Gao_2011} and are specified in \cref{fig:Phantom_attenuation} (middle)
together with their attenuation coefficients $g_i$, which are taken from the NIST dataset \url{www.nist.gov/pml/
	data/xraycoef/}.
The intensity $I_0$ emitted by the X-ray source per space angle is taken as $I_0(E)=\bar I i_0(E)$
for the profile $i_0$ shown in \cref{fig:Phantom_attenuation} right (taken from \cite{Ruth1997}) and $\bar I$ modelling the source strength and thereby the signal to noise ratio.

In our numerical experiments we actually discretize the photon energy range $[E_{\min},E_{\max}]$ by seven discrete energy levels
and the phantom by $64\times64$ pixels.

Measurements are obtained by computing the forward operator $\mathscr I(w)(E,\boldsymbol x,\boldsymbol\theta)$,
imposing Poisson noise to obtain $y(E,\boldsymbol x,\boldsymbol\theta)$,
integrating this $y$ in $E$ and finally adding Gaussian noise of different standard deviations $\sigma$.

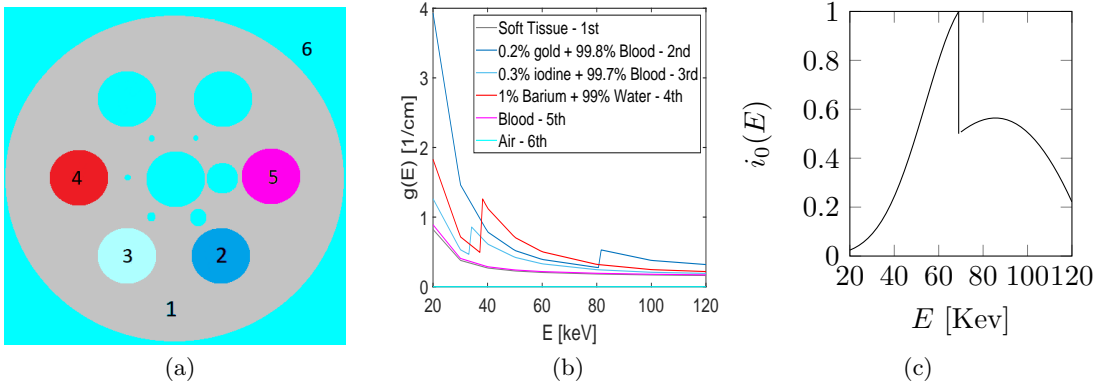
\begin{figure}[h]
	\centering
	
	\begin{tabular}{c  c  c }
			\subfloat[]{{\includegraphics[width=.3\linewidth]{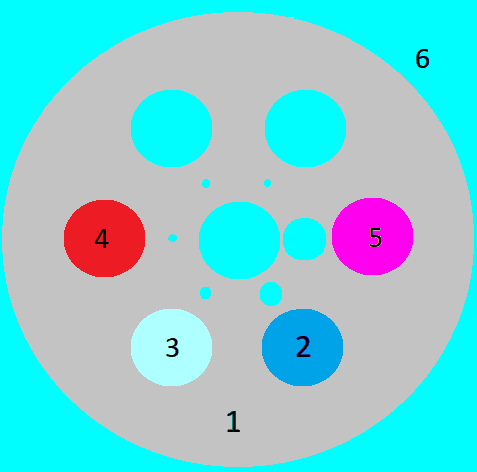} }} & \subfloat[]{{\includegraphics[width=.31\linewidth,height=.32\linewidth]{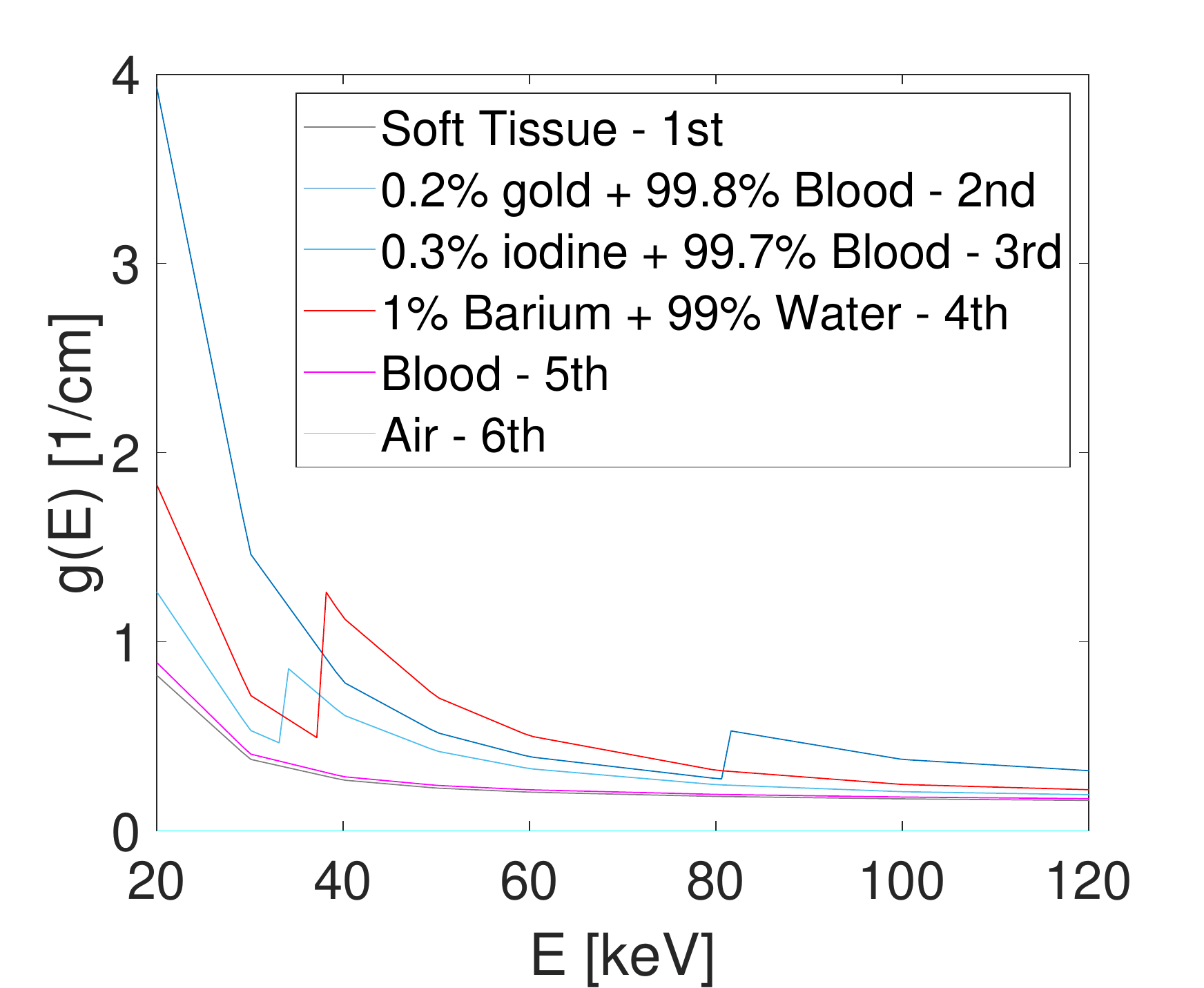} }} & \hspace{-0.3in}
			\subfloat[]{\begin{tikzpicture}
			\begin{axis}[xmin =20,xmax=120,ymin =0,ymax=1,
			width = .3\textwidth,
			height = .32\textwidth,
			xlabel = {${E}$ [Kev]},
			ylabel = {$i_0({E})$}]
			\addplot[domain=20:69] {x*exp(-(x-69)^2/2/22^2)/69};
			\addplot[black] coordinates {(69,1)(69,1/2)};
			\addplot[domain=70:120] {x*cos(deg((x-69)/(130-69)*pi/2))/2/69};
			\end{axis}
			\end{tikzpicture}}
			
	\end{tabular}
	\caption{Employed phantom (a) with material indicated by number, attenuation coefficients of the different materials (b), intensity profile of the X-ray source (c).
}
	\label{fig:Phantom_attenuation}
\end{figure}


\subsection{Reconstructions for different noise settings}

We will perform reconstructions for three different noise settings:
\begin{enumerate}
	\renewcommand{\theenumi}{\alph{enumi}}
	\item\label{enm:lowNoise} low Poisson and low Gaussian noise (high signal to noise ratio),
	\item\label{enm:middleNoise} high Poisson and low Gaussian noise (low X-ray source intensity),
	\item\label{enm:highNoise} high Poisson and high Gaussian noise (low X-ray source intensity \& detector quality).
\end{enumerate}
The standard deviation for Gaussian noise in the three cases is taken as $\sigma=2\cdot 10^{-3}, 2\cdot10^{-6}, 10^2$, respectively.
The X-ray source intensity in the cases is taken as $\bar I=3\cdot10^{11}, 1500$, and again $1500$.
Examples of the corresponding sinograms along with their histograms are shown in \cref{fig:CTdata}.
%

\begin{figure}
\centering
\setlength\unitlength{\linewidth}%
\setlength\tabcolsep{.01\linewidth}%
\begin{tabular}{c|c|c}
\includegraphics[width=0.15\linewidth]{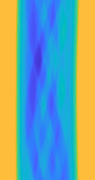}
\begin{tikzpicture}
\begin{axis}[ybar interval, ymax=6000, ymin=0, xmin=521426877054.930176, xmax=1087547486428.854248, ticks=none, yscale=0.75, xscale=.33, grid=none]
\addplot coordinates { (550000000000,17) (570000000000,99) (590000000000,197) (610000000000,494) (630000000000,588) (650000000000,797) (670000000000,1223) (690000000000,1203) (710000000000,829) (730000000000,838) (750000000000,1398) (770000000000,1199) (790000000000,664) (810000000000,369) (830000000000,307) (850000000000,166) (870000000000,267) (890000000000,76) (910000000000,125) (930000000000,158) (950000000000,55) (970000000000,81) (990000000000,5950) (1010000000000,0) };
\end{axis}
\end{tikzpicture}
&
\includegraphics[width=0.15\linewidth]{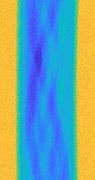}
\begin{tikzpicture}
\begin{axis}[ybar interval, ymax=6000, ymin=0, xmin=2607.134385, xmax=5437.737432, ticks=none, yscale=0.75, xscale=.33, grid=none]
\addplot coordinates { (2650,2) (2750,31) (2850,108) (2950,214) (3050,462) (3150,610) (3250,857) (3350,1148) (3450,1130) (3550,859) (3650,968) (3750,1262) (3850,1128) (3950,695) (4050,418) (4150,306) (4250,215) (4350,204) (4450,124) (4550,126) (4650,134) (4750,103) (4850,855) (4950,2962) (5050,1962) (5150,210) (5250,7) (5350,0) };
\end{axis}
\end{tikzpicture}
&
\includegraphics[width=0.15\linewidth]{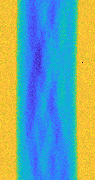}
\begin{tikzpicture}
\begin{axis}[ybar interval, ymax=6000, ymin=0, xmin=2607.134385, xmax=5437.737432, ticks=none, yscale=0.75, xscale=.33, grid=none]
\addplot coordinates { (2450,1) (2550,2) (2650,10) (2750,48) (2850,129) (2950,272) (3050,448) (3150,635) (3250,847) (3350,1030) (3450,1045) (3550,1015) (3650,1032) (3750,1073) (3850,1027) (3950,720) (4050,481) (4150,331) (4250,229) (4350,217) (4450,147) (4550,125) (4650,179) (4750,445) (4850,1201) (4950,1860) (5050,1618) (5150,726) (5250,179) (5350,26) (5450,2) (5550,0) };
\end{axis}
\end{tikzpicture}
\\
\makebox[0.15\linewidth]{}
\begin{picture}(0.15,0)
\put(0,.02){\includegraphics[width=0.15\linewidth,height=1ex]{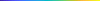}}
\put(-.01,0){\footnotesize$5.2\mathrm{e}{11}$}
\put(.09,0){\footnotesize$10.9\mathrm{e}{11}$}
\end{picture}&
\makebox[0.15\linewidth]{}
\begin{picture}(0.15,0)
\put(0,.02){\includegraphics[width=0.15\linewidth,height=1ex]{colorbar}}
\put(-.01,0){\footnotesize$2.6\mathrm{e}{3}$}
\put(.11,0){\footnotesize$5.4\mathrm{e}{3}$}
\end{picture}&
\makebox[0.15\linewidth]{}
\begin{picture}(0.15,0)
\put(0,.02){\includegraphics[width=0.15\linewidth,height=1ex]{colorbar}}
\put(-.01,0){\footnotesize$2.6\mathrm{e}{3}$}
\put(.11,0){\footnotesize$5.4\mathrm{e}{3}$}
\end{picture}
\\
noise \eqref{enm:lowNoise}&
noise \eqref{enm:middleNoise}&
noise \eqref{enm:highNoise}
\end{tabular}
	\caption{Measured sinograms for the three different noise settings along with the histograms of their values indicating the amount of spreading due to noise.
}
	\label{fig:CTdata}
\end{figure}



For each noise setting \eqref{enm:lowNoise}-\eqref{enm:highNoise} we run the reconstruction using our main algorithm from \cref{sec:FBSalg},
but we also test the two alternatives described in the appendix in \cref{sec:PDscheme,sec:ADMM}.
In the following we abbreviate the algorithms as EM (expectation maximization), PD (primal-dual), and ADMM (alternating direction of multipliers method).

For the EM algorithm we use the damping parameter $\omega = .1$ with maximum iteration number 23000 for the cases \eqref{enm:lowNoise}, \eqref{enm:highNoise} and 17000 else.
In the inner primal-dual iteration from \cref{sec:ROF} we pick step lengths $\tau = 0.02$, $\sigma=6.25$, overrelaxation $\theta = .1$, and iteration number $1000$.
The reinitialization from \cref{subsec:reint} is performed every 11000, 6000, and 10000 iterations, respectively.

For the PD algorithm we employ 500 maximum inner primal-dual iterations
with an overrelaxation $\theta=1$ and step sizes $\rho_1=\frac1{20}$, $\rho_2=\frac1{100}$, $\tau=\frac1{200}$.
For the three noise cases we require 50, 32, and 55 outer iterations,
and the reinitialization from \cref{subsec:reint} is performed every 15, 15 and 5 iterations, respectively.

For the ADMM we require 8000, 1500 and 1600 iterations in the different noise cases,
and the reinitialization from \cref{subsec:reint} is performed every 150  iterations for the noise case \eqref{enm:lowNoise} and every 350 iterations for \eqref{enm:middleNoise}, \eqref{enm:highNoise} respectively.
The Lagrange penalty weights $\mu_1,\mu_2$ are updated adaptively whenever $F(y^{k+1},\boldsymbol{\tilde w}^{k+1})>F(y^{k},\boldsymbol{\tilde w}^{k})$ by multiplication with $10$ when their values are less than 1e9, 1e5 and 1e1 in cases \eqref{enm:lowNoise}, \eqref{enm:middleNoise} and \eqref{enm:highNoise} respectively, by multiplication with $2$ when their values are between 1e9 and 1e13 or 1e5 and 1e13 or 1e1 and 1e8  in each case respectively 
and by division by 1000 whenever exceed  1e13 in cases \eqref{enm:lowNoise}, \eqref{enm:middleNoise} and 1e8 else. For updating the Lagrange penalty weight $\mu_3$ we follow the same strategy with the difference that whenever $\mu_3$ exceeds  1e9, 1e5 and 1e1 in each case respectively we divide it by 10.

The regularization parameter was chosen as $\alpha=10^6$, $0.1$, $0.07$ in the three noise settings
(note that for high X-ray source intensity the Kullback-Leibler data term becomes very strong so that $\alpha$ needs to be high in that case to have any influence),
while $\beta$ was actually set to zero experimentally:
It turned out that the above described regular reinitializations from \cref{subsec:reint} were sufficient and more efficient in guiding the reconstruction to pure material distributions.

The reconstructions for the three noise settings are provided in \Cref{fig:rec_4mat_low_noise_reinit}.
In noise case \eqref{enm:lowNoise} the reconstructions are very clean and almost equal the ground truth,
while small features (the smaller holes) are removed in the results for noise cases \eqref{enm:middleNoise} and \eqref{enm:highNoise}.
In addition, in these cases the different materials sometimes got mixed up, either in small subregions near the boundary between two materials
or for a complete material as in noise case \eqref{enm:highNoise}, where the PD reaches a minimum interpreting ground truth material 2 as material 4.

\Cref{fig:rec_5mat_low_noise_reinit} shows reconstructions for noise case \eqref{enm:lowNoise} with an additional material
taken to have almost the same attenuation behaviour as soft tissue to produce a highly challenging situation with two very similar materials.
All algorithms produce spurious pixels, where the effect is the strongest for the EM algorithm:
As already predicted in \cref{sec:FBSalg} it is quite slow in removing artefacts in the low noise case.

\begin{figure}
  \centering%
  \setlength\unitlength{.12\linewidth}%
  \setlength\tabcolsep{.04\unitlength}%
	\vspace*{-2\baselineskip}
	\begin{tabular}{l|l}
		&
		\begin{tabular}{lccccc}
			\begin{minipage}[b]{3.35em}
				ground truth\\
			\end{minipage}
			&\includegraphics[width=\unitlength]{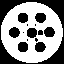}&\includegraphics[width=\unitlength]{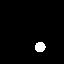}&\includegraphics[width=\unitlength]{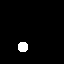}&\includegraphics[width=\unitlength]{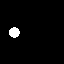}&\includegraphics[width=\unitlength]{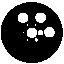}\\
		\end{tabular}\\
		\hline\\[-2ex]
		noise \eqref{enm:lowNoise}&
		\begin{tabular}{lccccc}
			\raisebox{.42\unitlength}{EM}&\includegraphics[width=\unitlength]{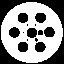}&\includegraphics[width=\unitlength]{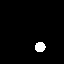}&\includegraphics[width=\unitlength]{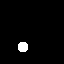}&\includegraphics[width=\unitlength]{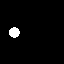}&\includegraphics[width=\unitlength]{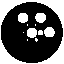}\\
			\raisebox{.42\unitlength}{PD}&\includegraphics[width=\unitlength]{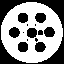}&\includegraphics[width=\unitlength]{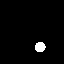}&\includegraphics[width=\unitlength]{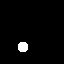}&\includegraphics[width=\unitlength]{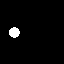}&\includegraphics[width=\unitlength]{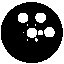}\\
			\raisebox{.42\unitlength}{ADMM}&\includegraphics[width=\unitlength]{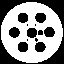}&\includegraphics[width=\unitlength]{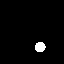}&\includegraphics[width=\unitlength]{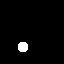}&\includegraphics[width=\unitlength]{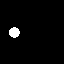}&\includegraphics[width=\unitlength]{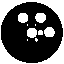}
		\end{tabular}\\
		\hline\\[-2ex]
		noise \eqref{enm:middleNoise}&
		\begin{tabular}{lccccc}
			\raisebox{.42\unitlength}{EM}&\includegraphics[width=\unitlength]{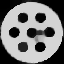}&\includegraphics[width=\unitlength]{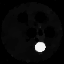}&\includegraphics[width=\unitlength]{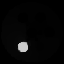}&\includegraphics[width=\unitlength]{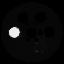}&\includegraphics[width=\unitlength]{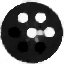}\\
			\raisebox{.42\unitlength}{PD}&\includegraphics[width=\unitlength]{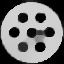}&\includegraphics[width=\unitlength]{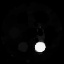}&\includegraphics[width=\unitlength]{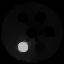}&\includegraphics[width=\unitlength]{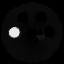}&\includegraphics[width=\unitlength]{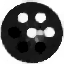}\\
			\raisebox{.42\unitlength}{ADMM}&\includegraphics[width=\unitlength]{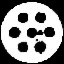}&\includegraphics[width=\unitlength]{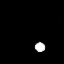}&\includegraphics[width=\unitlength]{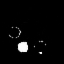}&\includegraphics[width=\unitlength]{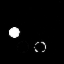}&\includegraphics[width=\unitlength]{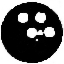}
		\end{tabular}\\
		\hline\\[-2ex]
		noise \eqref{enm:highNoise}&
		\begin{tabular}{lccccc}
			\raisebox{.42\unitlength}{EM}&\includegraphics[width=\unitlength]{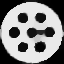}&\includegraphics[width=\unitlength]{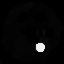}&\includegraphics[width=\unitlength]{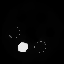}&\includegraphics[width=\unitlength]{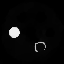}&\includegraphics[width=\unitlength]{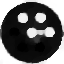}\\
			\raisebox{.42\unitlength}{PD}&\includegraphics[width=\unitlength]{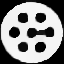}&\includegraphics[width=\unitlength]{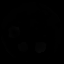}&\includegraphics[width=\unitlength]{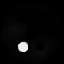}&\includegraphics[width=\unitlength]{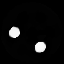}&\includegraphics[width=\unitlength]{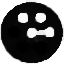}\\
			\raisebox{.42\unitlength}{ADMM}&\includegraphics[width=\unitlength]{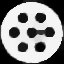}&\includegraphics[width=\unitlength]{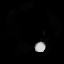}&\includegraphics[width=\unitlength]{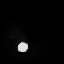}&\includegraphics[width=\unitlength]{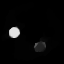}&\includegraphics[width=\unitlength]{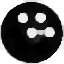}
		\end{tabular}
	\end{tabular}
	\caption{Reconstructed materials for the different noise settings and algorithms.}
	\label{fig:rec_4mat_low_noise_reinit}
\end{figure}

%

\begin{figure}
  \centering%
  \setlength\unitlength{1.8cm}%
  \setlength\tabcolsep{.04\unitlength}%
	\begin{tabular}{l|l}
		&
		\begin{tabular}{lcccccc}
			\begin{minipage}[b]{3.30em}
				ground truth\\
			\end{minipage}
			&\includegraphics[width=\unitlength]{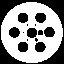}&\includegraphics[width=\unitlength]{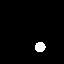}&\includegraphics[width=\unitlength]{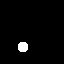}&\includegraphics[width=\unitlength]{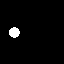}&\includegraphics[width=\unitlength]{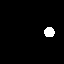}&\includegraphics[width=\unitlength]{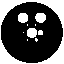}
		\end{tabular}\\
		\hline\\[-2ex]
		noise \eqref{enm:lowNoise}&
		\begin{tabular}{lcccccc}
			\raisebox{.42\unitlength}{EM}&\includegraphics[width=\unitlength]{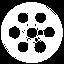}&\includegraphics[width=\unitlength]{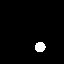}&\includegraphics[width=\unitlength]{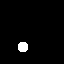}&\includegraphics[width=\unitlength]{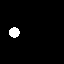}&\includegraphics[width=\unitlength]{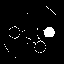}&\includegraphics[width=\unitlength]{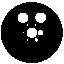}\\
			\raisebox{.42\unitlength}{PD}&\includegraphics[width=\unitlength]{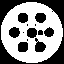}&\includegraphics[width=\unitlength]{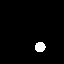}&\includegraphics[width=\unitlength]{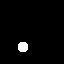}&\includegraphics[width=\unitlength]{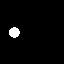}&\includegraphics[width=\unitlength]{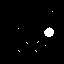}&\includegraphics[width=\unitlength]{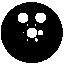}\\
			\raisebox{.42\unitlength}{ADMM}&\includegraphics[width=\unitlength]{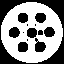}&\includegraphics[width=\unitlength]{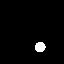}&\includegraphics[width=\unitlength]{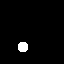}&\includegraphics[width=\unitlength]{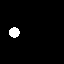}&\includegraphics[width=\unitlength]{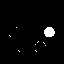}&\includegraphics[width=\unitlength]{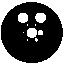}
		\end{tabular}
	\end{tabular}
	\caption{Reconstructed materials for the different noise settings and algorithms with one additional material (the fifth) of almost same attenuation characteristics as the first material. Despite the strong similarity between both materials they can be almost perfectly separated.}
	\label{fig:rec_5mat_low_noise_reinit}
\end{figure}

\section{Conclusion}
In general for inverse problems, an improved accuracy of forward modelling should lead to improved solutions.
We here covered three particular aspects that are important in several regimes of computed tomography:
the polychromaticity of the X-ray source
(which is particularly important if the attenuation varies strongly spatially and the X-ray dose is comparatively low),
the distinction between photon counting noise and electronic noise
(which is important for radiation doses in which the electronic noise becomes comparable to photon counting noise),
and the explicit modelling of different material compositions.
Despite the non-convexity, which is unavoidable in a model with the above ingredients,
the well-posedness of the model and the convergence of the proposed algorithm could be proven.
In addition, numerical examples on phantom data give satisfying results and validate the proposed model and algorithm.

\appendix

\section{An alternating minimization with primal-dual updates} \label{sec:PDscheme}
While the previous algorithm is quite fast and efficient for noisy data and high regularization, for low noise we need to change it.
We still employ the same update for $y$, but the update of $\boldsymbol w$ will be performed by the primal-dual iteration of Chambolle and Pock \cite{Chambolle2011}.

\subsection{Primal-dual update of $\boldsymbol w$}
To this end we write the $\boldsymbol w$-dependent part of the energy $F$ as
\begin{multline*}
\frac{F(y,\boldsymbol w)-\frac1{2\rho^2}\int_\Sigma\left(f-\int_{E_{\min}}^{E_{\max}}y\,\mathrm d\mathcal{L}^1\right)^2\,\mathrm d\mathcal H^n}\alpha
=F_1((\mathscr D,D)\boldsymbol w)+F_2(\boldsymbol w)
\quad\text{with}\\
\begin{aligned}
F_1(\boldsymbol u,\boldsymbol X)
&=F_{11}(\boldsymbol u)+F_{12}(\boldsymbol X)\\
&=\frac1\alpha\int_\Sigma\int_{E_{\min}}^{E_{\max}}d_{\mathrm{KL}}\left(y,I_0\exp\left(-\sum_{i=1}^{N_{\text{mat}}}g_iu_i\right)\right)\,\mathrm d\mathcal{L}^1\,\mathrm d\mathcal H^n
+\sum_{i=1}^{N_{\text{mat}}}\|X_i\|_{\mathcal M},\\
F_2(\boldsymbol w)
&=F_{21}(\boldsymbol w)+F_{22}(\boldsymbol w)\\
&=\frac\beta\alpha R_2(\boldsymbol w)+\int_\Omega\iota_\Delta(\boldsymbol w)\,\mathrm{d}\mathcal L^n,
\end{aligned}
\end{multline*}
where $\|X\|_{\mathcal M}$ denotes the total variation or norm of the vector-valued measure $X$.
The corresponding primal-dual iteration with dual variables
$\boldsymbol\phi=(\phi_1,\ldots,\phi_{N_{\text{mat}}}):\Sigma\to\R^{N_{\text{mat}}}$ and $\psi=(\psi_1,\ldots,\psi_{N_{\text{mat}}}):\Sigma\to(\R^2)^{N_{\text{mat}}}$,
step sizes $\tau>0$ and $\rho_1,\rho_2>0$ for primal descent and dual ascent
as well as relaxation parameter $\theta\in[0,1]$ reads
\begin{align*}
\boldsymbol\phi^{k+1}
&={\prox}_{\rho_1 F_{11}^*}\left(\boldsymbol\phi^{k}+\rho_1\mathscr D\boldsymbol{\overline w}^k\right),\\
\boldsymbol\psi^{k+1}
&={\prox}_{\rho_2 F_{12}^*}\left(\boldsymbol\psi^{k}+\rho_2D\boldsymbol{\overline w}^k\right),\\
\boldsymbol w^{k+1}
&={\prox}_{\tau F_{22}}\left(\boldsymbol w^k-\tau(\mathscr D^*\boldsymbol\phi^{k+1}+D^*\boldsymbol\psi^{k+1})-\tau F_{21}'(\boldsymbol w^k)\right),\\
\boldsymbol{\overline w}^{k+1}
&=\boldsymbol w^{k+1}+\theta(\boldsymbol w^{k+1}-\boldsymbol w^k).
\end{align*}
Note that $D^*=-\div$.
To ensure stability, the step sizes have to be chosen small enough with respect to the norms of the linear operators.
A feasible choice from \cite{champolleprec2011}, which we employ in our implementation, is
$\rho_1=\frac12$ (the reciprocal of the absolute row sum of the discretized derivative operator $D$),
$\rho_2=\frac1{\mathrm{diam}\Omega}$ (the reciprocal of the maximum absolute row sum of the discretized beam transform $\mathscr D$), and
$\tau=\frac1{2+\pi}$ (the reciprocal of the maximum absolute row sum of the discretization of $\mathscr D^*+D^*$).
The proximal operator ${\prox}_{\tau F_{22}}$ merely is the pointwise Euclidean projection onto the probability simplex $\Delta$,
which was detailed in \cref{sec:simplexProjection} (this time the weights $r_i$ equal $1$).
As for ${\prox}_{\rho_2 F_{12}^*}$ it is straightforward to compute the Legendre--Fenchel conjugate
$F_{12}^*(\boldsymbol\psi)=0$ if $|\psi_i|\leq1$ almost everywhere on $\Sigma$ for all $i$ and $F_{12}^*(\boldsymbol\psi)=\infty$ else.
Thus the corresponding proxial operator reads
\begin{equation*}
{\prox}_{\rho_2 F_{12}^*}(\boldsymbol\psi)(\boldsymbol x)
=\left(\frac{\psi_i(\boldsymbol x)}{|\psi_i(\boldsymbol x)|}\min\{1,|\psi_i(\boldsymbol x)|\}\right)_{i=1,\ldots,N_{\text{mat}}}.
\end{equation*}
As for the proximal operator of $F_{11}^*$ we exploit the equivalence
\begin{equation*}
\boldsymbol\xi={\prox}_{\rho_1 F_{11}^*}(\boldsymbol\phi)
\quad\Leftrightarrow\quad
\boldsymbol\xi\in\partial F_{11}\left(\frac{\boldsymbol\phi-\boldsymbol\xi}{\rho_1}\right)
=\frac1\alpha\int_{E_{\min}}^{E_{\max}}y\boldsymbol g-I_0\boldsymbol g\exp\left(\sum_{i=1}^{N_{\text{mat}}}\frac{\xi_i-\phi_i}{\rho_1}g_i\right)\,\mathrm d\mathcal L^1,
\end{equation*}
where for ease of notation we identified the subdifferential with its single element.
This equation can be solved for $\boldsymbol\xi$ separately and independently at all $(\boldsymbol x,\boldsymbol\theta)\in\Sigma$.
Indeed, abbreviating for fixed $(\boldsymbol x,\boldsymbol\theta)\in\Sigma$
\begin{align*}
B&=\frac{\boldsymbol g}{\rho_1}:(E_{\min},E_{\max})\to\R^{N_{\text{mat}}},\\
s(\boldsymbol x,\boldsymbol\theta)&=\frac{\rho_1}\alpha I_0\exp\left(-\frac1{\rho_1}\sum_{i=1}^{N_{\text{mat}}}\phi_i(\boldsymbol x,\boldsymbol\theta)g_i\right):(E_{\min},E_{\max})\to\R,\\
y&=\frac{\rho_1}\alpha y(\boldsymbol x,\boldsymbol\theta):\Sigma\to\R,
\end{align*}
it can readily be checked that the sought $\xi(\boldsymbol x,\boldsymbol\theta)\in\R^{N_{\text{mat}}}$ is the unique minimizer of the convex functional
\begin{equation*}
\frac12\left|\xi(\boldsymbol x,\boldsymbol\theta)-\int_{E_{\min}}^{E_{\max}}yB\,\mathrm{d}\mathcal L^1\right|^2
+\int_{E_{\min}}^{E_{\max}}s(\boldsymbol x,\boldsymbol\theta)\exp(B\cdot\xi(\boldsymbol x,\boldsymbol\theta))\,\mathrm{d}\mathcal L^1.
\end{equation*}
Discretizing the integral with numerical quadrature, this turns into an expression of the form
\begin{equation}\label{eqn:proxMatEq}
\frac12\left|\xi(\boldsymbol x,\boldsymbol\theta)-B^Ty\right|^2+s^T\exp(B\xi(\boldsymbol x,\boldsymbol\theta)),
\end{equation}
where the exponential is applied componentwise
and for simplicity we used the same symbols for the matrices $s,B$ and the vector $y$ resulting from discretizing the corresponding continuous objects.
Thus we obtain
\begin{equation*}
\xi(\boldsymbol x,\boldsymbol\theta)
={\prox}_{h}(B^Ty)
\qquad\text{for }
h(\boldsymbol\xi)=s^T\exp(B\xi).
\end{equation*}
The computation of this proximal operator is detailed in the next paragraph.

\subsection{Proximal operator for vectors of exponentials}
The minimizer $\boldsymbol\xi\in\R^{N_{\text{mat}}}$ of \eqref{eqn:proxMatEq} is computed via the corresponding Newton iteration
\begin{equation*}
\boldsymbol\xi^{n+1}=\boldsymbol\xi^n-(I+B^T\mathrm{diag}(\exp(B\boldsymbol\xi^n+\log s))B)^{-1}(\boldsymbol\xi^n+B^T\exp(B\boldsymbol\xi^n+\log s)-B^Ty),
\end{equation*}
where exponential and logarithm are applied componentwise and $\mathrm{diag}(v)$ for a vector $v$ indicates the diagonal matrix with diagonal $v$.
Note that for numerical stability it turns out to be important to add $\log s$ inside the exponential rather than multiplying with $s$. 
As initialization we pick
\begin{equation*}
\boldsymbol\xi^1=B^\dagger(\log y-\max\{\log s,-\delta\})
\end{equation*}
in which $B^\dagger$ denotes the Moore--Penrose inverse of $B$ and $\delta$ is a fixed parameter introduced to safeguard the case of vanishing $s$.
In our simulations we take $\delta=1000$.
Alternatively one can also initialize with the optimal $\boldsymbol\xi$ from the previous primal-dual iteration.
The reasoning for this initialization is that the solution $\boldsymbol\xi$ must be very small in absolute value due to the exponential amplification;
thus the optimality condition is approximately $B^T\exp(B\boldsymbol\xi+\log s)-B^Ty=0$ of which $\boldsymbol\xi^1$ is an approximate solution.

\section{ADMM approach}\label{sec:ADMM}
Yet another alternative, which does not require any inner iterations, is the following implementation via the alternating direction method of multipliers (ADMM).
To this end we rewrite the minimization of $F$ as
\begin{multline*}
\min\{ \mathscr{R}(\boldsymbol{X},\boldsymbol{\tilde w})+\mathscr H(y,\boldsymbol{z})\,|\,
y:[E_{\min},E_{\max}]\times\Sigma\to\R,
\boldsymbol{w}:\Omega\to\R^{N_{\text{mat}}},\\
\hspace*{20ex}\boldsymbol{X}=\nabla \boldsymbol{w}, \ \boldsymbol{\tilde w}=\boldsymbol{w}, \ \boldsymbol{z}=\mathscr{D}\boldsymbol{w}\},
\text{ where}\\
\begin{aligned}
\mathscr R(\boldsymbol{X},\boldsymbol{\tilde w})
&=\alpha\sum_{i=1}^{N_{\text{mat}}}|X_i|(\Omega) +\int_{\Omega}\iota_{\Delta}(\boldsymbol{\tilde w})\mathrm{d}\mathcal{L}^n+\beta R_2(\boldsymbol{\tilde w}),\\
\mathscr H(y,\boldsymbol{z})
&=\frac{1}{2\sigma^2}\int_{\Sigma}\left(f-\int_{E_{\min}}^{E_{\max}}y\mathrm{d}\mathcal{L}^1\right)^2\mathrm{d}\mathcal{H}^n
+\int_{\Sigma}\int_{E{\min}}^{E_{\max}}d_{KL}\left(y,\hat{\mathscr{I}}( \boldsymbol{z})\right)\mathrm{d}\mathcal{L}^1\mathrm{d}\mathcal{H}^n
\end{aligned}
\end{multline*}
for $\hat{\mathscr I}(\boldsymbol z)=I_0\exp(-\sum_{i=1}^{N_{\text{mat}}}g_iz_i)$.
We form the augmented Lagrangian
\begin{multline*}
\mathscr{L}(y,\boldsymbol{w},\boldsymbol{\tilde w},\boldsymbol{X},\boldsymbol{z};\ \boldsymbol{\Lambda},\boldsymbol{\lambda}_1,\boldsymbol{\lambda}_2) =\mathscr{R}(\boldsymbol{X},\boldsymbol{\tilde w})+\mathscr H(y,\boldsymbol{z}) + \langle \boldsymbol{\Lambda},\boldsymbol{X}-\nabla\boldsymbol{w}\rangle +\frac{\mu_1}{2}\|\boldsymbol{X}-\nabla\boldsymbol{w}\|^2_{L^2}\\
+\langle \boldsymbol{\lambda}_1,\boldsymbol{\tilde w}-\boldsymbol{w}\rangle +\frac{\mu_2}{2}\|\boldsymbol{\tilde w}-\boldsymbol{w}\|^2_{L^2} +\langle \boldsymbol{\lambda}_2,\boldsymbol{z}-\mathscr{D}\boldsymbol{w}\rangle +\frac{\mu_3}{2}\|\boldsymbol{z}-\mathscr{D}\boldsymbol{w}\|^2_{L^2}
\end{multline*}
with Lagrange multipliers $\boldsymbol{\Lambda}:\Omega \to \R^{ N_{\text{mat}}\times n}$, $\boldsymbol{\lambda}_1:\Omega \to \R^{ N_{\text{mat}}}$ and $\boldsymbol{\lambda}_2:\Sigma \to \R^{ N_{\text{mat}}}$
and penalty weights $\mu_1,\mu_2,\mu_3>0$.
The minimization problem is solved by alternatingly minimizing the augmented Lagrangian
with respect to the primal variables $\boldsymbol{w}$, $\boldsymbol{\tilde w}$, $\boldsymbol X$, $y$, and $\boldsymbol{z}$,
and then updating the Lagrange multipliers $\boldsymbol\Lambda$, $\boldsymbol{\lambda}_1$ and $\boldsymbol{\lambda}_2$, known as ADMM,
\begin{align}
\boldsymbol{w}^{k}
&= \argmin_{\boldsymbol{w}}\mathscr{L}(y^{k-1},\boldsymbol{w},\boldsymbol{\tilde{w}}^{k-1},\boldsymbol {X}^{k-1},\boldsymbol{z}^{k-1};\boldsymbol{\Lambda}^{k-1},\boldsymbol{\lambda}_1^{k-1},\boldsymbol{\lambda}^{k-1}_2),\\
%
%
(\boldsymbol{\tilde{w}}^{k},\boldsymbol X^{k},y^{k},\boldsymbol{z}^{k})
&= \argmin_{(\boldsymbol{\tilde{w}},\boldsymbol X,y,\boldsymbol z)}\mathscr{L}(y,\boldsymbol{w}^{k},\boldsymbol{\tilde{w}},\boldsymbol{X},\boldsymbol{z};\boldsymbol{\Lambda}^{k-1},\boldsymbol{\lambda}_1^{k-1},\boldsymbol{\lambda}^{k-1}_2),\\
\label{eqn:General_ADMM_ROF_Lambda}
(\boldsymbol\Lambda^{k},\boldsymbol{\lambda}^{k}_1,\boldsymbol{\lambda}_2^{k})
&=(\boldsymbol\Lambda^{k\!-\!1}\!\!+\!\mu_1(\boldsymbol X^{k} \!-\! \nabla\boldsymbol{w}^{k}),
\boldsymbol{\lambda}^{k\!-\!1}_1\!\!+\!\mu_2({\boldsymbol{\tilde{w}}}^{k} \!-\! \boldsymbol{w}^{k}),
\boldsymbol{\lambda}_2^{k\!-\!1}\!\!+\!\mu_3(\boldsymbol{z}^{k} \!-\! \mathscr{D}\boldsymbol{w}^{k})).
%
\end{align}
The solution $\boldsymbol w^k$ can readily be shown to satisfy the linear optimality condition
\begin{equation}\label{eq:linsystw}
(\mu_2 I +\mu_3\mathscr{D}^*\mathscr{D} -\mu_1\Delta)w^{k}_i = \lambda^{k-1}_{1,i} +\mu_2 \tilde{w}^{k-1}_i +\mathscr{D}^*\!\!\left(\lambda^{k-1}_{2,i}+\mu_3z^{k-1}_i\right)-\div(\Lambda^{k-1}_i +\mu_1 X^{k-1}_i)
\end{equation}
for $i=1,\ldots,N_{\text{mat}}$ (where we assume periodic boundary conditions for simplicity and $I$ represents the identity operator).
This linear system of equations is solved using the preconditioned generalized minimal residual method with preconditioner
$$P= \mathcal{F}^{-1}\left(\mu_2 \mathcal{F}(I)+\mu_3c\mathcal{F}\left(\left(-\Delta\right)^{-\frac{1}2}\right)-\mu_1 \mathcal{F}(\Delta)\right).$$
Above, $\mathcal F(A)$ denotes the representation of a linear operator $A$ in Fourier space.
Since $-\Delta$ is diagonal in Fourier space, $P$ is straightforward and highly efficient to invert.
The choice of the term $(-\Delta)^{-1/2}$ is motivated by $\mathscr{D}^*\mathscr{D}$ equalling that linear operator on an unbounded domain.
The components $X_i^k$ of $\boldsymbol X^k$ are readily computed for $i=1,\ldots,N_{\text{mat}}$ via
\begin{equation}\label{eqn:ADMM_ROF_XUpdatem}
X_i^{k}(x)
=\prox\hfil_{\frac{\alpha|\cdot|}{\mu_1}}\left(\nabla w_i^{k}(x)-\frac{\Lambda_i^{k-1}(x)}{\mu_1}\right)
\;\text{with }
\prox\hfil_{\frac{\alpha|\cdot|}{\mu_1}}\boldsymbol v
=\max\{0,|\boldsymbol v|-\tfrac{\alpha}{\mu_1}\}\tfrac{\boldsymbol v}{|\boldsymbol v|}.
\end{equation}
For given $\boldsymbol z^k$, the optimal value $y^k$ is due to \cref{thm:optimalY} given as $y^k(x)=y(\boldsymbol z^k(x))$ for
\begin{equation}\label{eqn:yOfZ}
y(\boldsymbol{z})
=\hat{\mathscr{I}}(\boldsymbol z)\exp\left(\tfrac f{\sigma^2}-W_0\left(\hat{\mathscr{F}}(\boldsymbol z)\exp(f/\sigma^2)\right)\right)
=\hat{\mathscr{I}}(\boldsymbol z)\tfrac{W_0\left(\hat{\mathscr{F}}(\boldsymbol z)\frac{\exp(f/\sigma^2)}{\sigma^2}\right)}{\hat{\mathscr{F}}(\boldsymbol z)/\sigma^2},
\end{equation}
where we abbreviated $\hat{\mathscr F}(\boldsymbol z)=\int_{E_{\min}}^{E_{\max}}\hat{\mathscr I}(\boldsymbol z)\,\mathrm{d}E$.
Having thus eliminated $y^k$, it still remains to minimize for $\boldsymbol z^k$.
Defining
$H^k(y,\boldsymbol z)=\mathscr{L}(y,\boldsymbol{w}^{k},\boldsymbol{\tilde{w}},\boldsymbol{X},\boldsymbol{z};\boldsymbol{\Lambda}^{k-1},\boldsymbol{\lambda}_1^{k-1},\boldsymbol{\lambda}^{k-1}_2)$,
the optimality condition for $\boldsymbol z^k$ becomes
\begin{multline*}
0
=\tfrac{\mathrm d}{\mathrm d\boldsymbol z}H^k(y(\boldsymbol z^k),\boldsymbol z^k)
=\partial_{y}H^k(y(\boldsymbol z^k),\boldsymbol z^k)\tfrac{\mathrm dy(\boldsymbol z^k)}{\mathrm d\boldsymbol z}+\partial_{\boldsymbol z}H^k(y(\boldsymbol z^k),\boldsymbol z^k)
=\partial_{\boldsymbol z}H^k(y(\boldsymbol z^k),\boldsymbol z^k)\\
=\boldsymbol\lambda_2+\mu_3(\boldsymbol z^k-\mathscr{D}\boldsymbol w^{k})
+\int_{E_{\min}}^{E_{\max}}\left(y(\boldsymbol{z}^k)-\hat{\mathscr I}(\boldsymbol{z}^k)\right)\boldsymbol g\,\mathrm{d}\mathcal{L}^1
=:h(\boldsymbol z),
\end{multline*}
which is solved for $\boldsymbol z^k$ pointwise on $\Sigma$ by a rapidly converging (damped) Newton iteration
\begin{equation}\label{Alg:Damped_Newtonzy}
\boldsymbol{z}^{k,\text{new}} = \boldsymbol{z}^{k,\text{old}}\!-\!\gamma^k(Dh(\boldsymbol{z}^{k,\text{old}}))^{-1}h(\boldsymbol{z}^{k,\text{old}})
\;\text{with }
\gamma^k  = \min\left(1,\tfrac{10}{\|Dh^{-1}(\boldsymbol{z}^{k,\text{old}})h(\boldsymbol{z}^{k,\text{old}})\|_2}\right).
\end{equation}
For the reader's convenience we provide the Newton operator
\begin{multline*}
\nabla h(\boldsymbol{z}) = \mu_3 - \left( \frac{W_0\left(\frac{\exp (f/\sigma^2)}{\sigma^2}\hat{\mathscr F}(\boldsymbol{z})\right)}{\hat{\mathscr F}(\boldsymbol{z})/\sigma^2}-1\right)\int_{E_{\min}}^{E_{\max}}\boldsymbol{g}\otimes\boldsymbol{g}\,\hat{\mathscr I}( \boldsymbol{z})\,\mathrm{d}\mathcal{L}^1 \\
+\frac{W^2_0\left(\frac{\exp(f/\sigma^2)}{\sigma^2}\hat{\mathscr F}\left(\boldsymbol{z}\right)\right)}{\hat{\mathscr F}\left(\boldsymbol{z}\right)^2\left(1+W_0\left(\frac{\exp (f/\sigma^2)}{\sigma^2}\hat{\mathscr F}(\boldsymbol{z})\right)\right)/\sigma^2} \int_{E_{\min}}^{E_{\max}}\boldsymbol{g}\,\hat{\mathscr I}(\boldsymbol{z})\,\mathrm{d}\mathcal{L}^1\otimes\int_{E_{\min}}^{E_{\max}}\boldsymbol{g}\,\hat{\mathscr I}(\boldsymbol{z})\,\mathrm{d}\mathcal{L}^1.
\end{multline*}
Finally, the optimization problem in $\boldsymbol{\tilde w}$ is nonconvex,
so instead of a full optimization we actually replace the nonconvex term $R_2(\boldsymbol{\tilde w})$ with its first order Taylor expansion at $\boldsymbol{\tilde w}^{k-1}$
(which corresponds to just doing an explicit gradient step in the nonconvex term) and thereby obtain
\begin{equation}\label{eq:subpronforw_tild_new1}
\boldsymbol{\tilde w}^{k}(\boldsymbol x)
=\argmin_{\boldsymbol{\tilde w}(\boldsymbol{x})} \frac{\mu_2}{2}\left|\boldsymbol{\tilde{w}}(\boldsymbol{x}) \!-\! \frac{\mu_2\boldsymbol w^{k}(\boldsymbol{x}) \!-\! \boldsymbol\lambda^{k\!-\!1}_1(\boldsymbol{x})\!+\!\beta(\boldsymbol{\tilde{w}}^{k\!-\!1}(\boldsymbol{x}) \!-\! \frac{\boldsymbol{1}_{N_{\text{mat}}}}{N_{\text{mat}}})}{\mu_2}\right|^2
\!+\!\iota_{\Delta}(\boldsymbol{\tilde{w}}(\boldsymbol x)),
\end{equation}
whose solution was detailed in \cref{sec:simplexProjection}.
The full scheme is summarized in \cref{Alg:FULL_ADMM}.


\begin{algorithm}
	\caption{ADMM for poly-energetic reconstruction }
	\label{Alg:FULL_ADMM}
	\begin{algorithmic}
		\Repeat
		\State calculate $w_i^{k}$ via solving the linear system \eqref{eq:linsystw} with preconditioned GMRES
		\State calculate $X_i^{k}$ via \eqref{eqn:ADMM_ROF_XUpdatem}
		\State calculate $\boldsymbol z^{k}$ via the damped Newton method \eqref{Alg:Damped_Newtonzy} 
		\State calculate $y^{k}=y(\boldsymbol z^{k})$ by \eqref{eqn:yOfZ}
		\State calculate $\boldsymbol{\tilde{w}}^{k}$ via solving \eqref{eq:subpronforw_tild_new1} according to \cref{sec:simplexProjection}
		\State calculate the dual variables $\boldsymbol\Lambda^{k}$, $\boldsymbol\lambda_1^{k}$ and $\boldsymbol\lambda_2^{k}$ via  \eqref{eqn:General_ADMM_ROF_Lambda}
		\State $k\leftarrow k+1$
		\Until $\|\boldsymbol {\tilde{w}}^{k}-\boldsymbol {\tilde{w}}^{k-1}\|_{L^2}/\|\boldsymbol {\tilde{w}}^{k-1}\|_{L^2}<\text{tolerance}$
	\end{algorithmic}
\end{algorithm}


\newpage
\bibliographystyle{siam}
\bibliography{paper_Georgios_Papanikos_Wirth_Benedikt}

\end{document}